\documentclass[10pt,a4paper,reqno]{amsart}
\usepackage{amsmath, amsthm, amsfonts, amssymb, mathrsfs}
\usepackage{graphicx}
\usepackage{mathtools}
\usepackage[update,prepend]{epstopdf} 
\usepackage{color}
\usepackage{xcolor}
\usepackage{soul}
\usepackage{cite}
\usepackage{subcaption}

\numberwithin{equation}{section}

\usepackage{float}
\usepackage{marginnote}
\usepackage{enumerate}
\usepackage{marginnote}
\usepackage{hyperref}
\newcommand{\ie}{\emph{i.e.}}
\newcommand{\eg}{\emph{e.g.}}

\newcommand{\sgn}{\mathop{\mathrm{sgn}}\nolimits}

\newcommand{\dist}{\mathop{\mathrm{dist}}\nolimits}

\newcommand{\ups}{\upsilon}
\newcommand{\la}{\lambda}
\newcommand{\RR}{{\mathbb{R}}}
\newcommand{\CC}{{\mathbb{C}}}
\newcommand{\ZZ}{{\mathbb{Z}}}

\newcommand{\ii}{{\rm i}}
\newtheorem{Theorem}{Theorem}

\newtheorem{Lemma}{Lemma}
\theoremstyle{definition}

\newtheorem{Remark}{Remark}

\usepackage{color}

\definecolor{DarkBlue}{rgb}{0,0.1,0.7}
\definecolor{DarkGreen}{rgb}{0,0.5,0.1}

\newcommand{\Hm}[1]{\leavevmode{\marginpar{\tiny%
			$\hbox to 0mm{\hspace*{-0.5mm}$\leftarrow$\hss}%
			\vcenter{\vrule depth 0.1mm height 0.1mm width \the\marginparwidth}%
			\hbox to
			0mm{\hss$\rightarrow$\hspace*{-0.5mm}}$\\\relax\raggedright #1}}}

\title[Spectral bounds for non-self-adjoint discrete 
Schr\"odinger operators]
{Spectral enclosures for non-self-adjoint \\discrete 
	Schr\"odinger operators}
\begin{document}
\author{Orif \,O.\ Ibrogimov}
\address[O.~O.~Ibrogimov]{Department of Mathematics, Faculty 
	of Nuclear Sciences and Physical 
	Engineering, Czech Technical University in Prague,
	Trojanova 13, 12000 Praha~2, Czech 
	Republic}
%\curraddr{}
\email{ibrogori@fjfi.cvut.cz}
\author{Franti{\v s}ek {\v S}tampach}
\address[F.~{\v S}tampach]{
	Department of Applied Mathematics, Faculty of Information 
	Technology, Czech Technical University in~Prague, 
	Th{\' a}kurova~9, 160~00 Praha, Czech Republic
}
%\curraddr{}
\email{stampfra@fit.cvut.cz}

\date{\today}
%=============================================================%
\begin{abstract}
We study location of eigenvalues of one-dimensional 
discrete Schr\"odinger operators with complex $\ell^{p}$-potentials 
for $1\leq p\leq \infty$. In the case of $\ell^{1}$-potentials, the 
derived bound is shown to be optimal. For $p>1$, two different spectral bounds are obtained. The method relies on the Birman--Schwinger principle and various techniques for estimations of the norm of the Birman--Schwinger operator. 
\end{abstract}
\maketitle
%
%=============================================================%
\section{Introduction and main results}
%=============================================================%
\noindent 

Recent years have seen a significant development in the spectral theory of
non-self-adjoint operators. A great deal of research aims to
a localization of spectra of differential operators such as Schr{\"o}dinger 
operators~\cite{Abr-Asl-Dav-01,Boegli-CMP-17,Cue-Ken-CPDE17,Dav-Nath02,Enblom-LMP-16,Fan-Kre-Veg-JST18,Frank-BLMS-11,Fra-Sim-JST17,Frank-TAMS-18,Lap-Saf-CMP09,Lee-Seo-JMAA-19,Safr-BLMS10, Fra-Lap-Sei-11, Fan-Kre-Veg-JFA-2018, Hen-Kre-JST2017, Krej-Siegl-JFA2019}, Dirac operators~\cite{Cuenin-IEOT14,Cue-Lap-Tre-14,Fan-Kre-LMP19} and others~\cite{Cuenin-IEOT-17,Ibr-Kre-Lap, Hans-LMP11}. For related recent results obtained in an abstract operator theoretic setting, the reader may also consult~\cite{Beh-Lan-Lot-Roh-JFA-18, Cue-Tre-JMAA16,Dem-Hans-Kat-JFA-09,Dem-Hans-Kat13}.

In contrast to the differential operators, there are almost no works studying similar questions for their discrete analogues, \ie~the difference operators like discrete Schr{\"o}dinger or discrete Dirac operators with complex potentials. The authors are 
only aware of~\cite{Hulko-BMS-17} which is focused on the number rather than the location of eigenvalues for the discrete Schr{\" o}dinger operator with a complex potential and a related work \cite{Kor-Lap-18} for the cubic lattice. Some constrains on the location of the discrete spectrum of semi-infinite complex Jacobi matrices are discussed in~\cite{Ego-Gol-05}.

Our main interest is to bridge this gap focusing first on the location of spectrum of the one-dimensional dicrete Schr{\" o}dinger operator with a complex valued $\ell^{p}$-potential, $1\leq p\leq\infty$, which is the objective of this paper. The one-dimensional Dirac operator with a complex potential is discussed in a separate paper~\cite{Cas-Ibr-Kre-Sta-inprep}.

Let $H_{0}$ be the discrete Laplacian acting on the 
Hilbert space~$\ell^{2}(\ZZ)$ which is the bounded operator 
determined by the equation
\[
H_{0}e_{n}:=e_{n-1}+e_{n+1}, \quad \forall n\in\ZZ,
\]
where $\{e_{n}\}_{n\in\ZZ}$ stands for the standard basis 
of~$\ell^{2}(\ZZ)$. Further, to a given complex 
sequence $\ups=\{\ups_{n}\}_{n\in\ZZ}$, we define the potential~$V$ 
as a diagonal operator 
\[
Ve_{n}:=\ups_{n}e_{n}, \quad \forall n\in\ZZ.
\]
The matrix representation of the discrete Schr\"{o}dinger 
operator $H_V:=H_0+V$ is the doubly-infinite Jacobi matrix
\[
	H_V=\begin{pmatrix}
	\ddots & \ddots & \ddots &        &        & &       \\
	&  1     & \ups_{-1} & 1      &        & &       \\
	&        & 1      & \ups_{0}  & 1      & &     \\
	&        &        &  1     & \ups_{1}  & 1 &     \\       &        &        &        & \ddots & \ddots & \ddots\\
	\end{pmatrix}.
\]
It is well known that the spectrum of the unperturbed operator $H_0$ is absolutely 
continuous and covers the interval $[-2,2]$. 
If the potential $V$ vanishes at infinity, \ie~$\ups_{n}\to0$ as
$n\to\pm\infty$, then $V$ is compact and the essential spectrum 
of the perturbed operator $H_V=H_0+V$ coincides 
with the interval $[-2,2]$. 
%Here we use any definition of the essential spectrum which is invariant under compact perturbations, see \cite[Sec. IX.2]{Edm-Eva-87} for details. 

The goal of the present paper is to investigate the location of the spectrum of $H_{V}$ with an $\ell^{p}$--potential $V$ for $1\leq p\leq\infty$. Our primary result concerns the location of eigenvalues of $H_{V}$ with an $\ell^{1}$--potential. Moreover, the obtained bound is shown to be optimal; see Theorem~\ref{thm:1-norm}. If the decay of the potential is slower, namely if $v\in\ell^{p}(\ZZ)$ for $1<p\leq\infty$, we derive bounds for the entire spectrum of $H_{V}$; see Theorems~\ref{thm:p-norm} and~\ref{thm:p-norm.2}.

\begin{Theorem}\label{thm:1-norm}
Let $\ups\in\ell^{1}(\ZZ)$. Then 
\begin{equation}\label{eq:1-norm}
	\sigma_{\rm{p}}(H_V)\subset	\Big\{\la\in\CC\setminus(-2,2) \;\big|\;\; |\la^2-4|\leq \|\ups\|^2_{\ell^1(\ZZ)}\Big\}.
\end{equation}
\end{Theorem}
\noindent
The bound in \eqref{eq:1-norm} is sharp in the following sense: For any $Q>0$ and any point~$\la\in\CC\setminus(-2,2)$ which fulfills the equation
\[
 |\la^2-4|=Q^{2},
\]
there exists $\ups\in\ell^{1}(\ZZ)$ such that $Q=\|\ups\|_{\ell^1(\ZZ)}$ and $\lambda$
is an eigenvalue of the corresponding discrete Schr\"{o}dinger operator $H_{V}$, see Section~\ref{Sect.Ex}. This means that every boundary point of the spectral enclosure~\eqref{eq:1-norm}, with the exception of points located in~$(-2,2)$, is an eigenvalue of~$H_V$ for some $\ell^{1}$--potential~$V$. Hence the obtained spectral bound cannot be squeezed any further. 

Clearly, the geometry of the spectral enclosure \eqref{eq:1-norm}
depends on the $\ell^1$-norm of the potential. If $\|\ups\|_{\ell^{1}(\ZZ)}\leq2$, the spectral enclosure consists of two connected components each containing either the point $2$ or $-2$. If $\|\ups\|_{\ell^{1}(\ZZ)}>2$, the spectral enclosure is a connected set containing the entire interval $[-2,2]$. The set is always symmetric with respect to both the real and the imaginary axes, see Figure~\ref{fig:indef} in Section~\ref{sec:plots}.

Our next result provides a spectral estimate in terms of 
the $\ell^p$-norm of the potential. 
In the sequel, for $p\in(1,\infty]$, we denote by $q\in[1,\infty)$ the corresponding H\"older exponent, \ie~$q=p/(p-1)$ if $1<p<\infty$ and 
$q=1$ if $p=\infty$.

\begin{Theorem}\label{thm:p-norm}
Let $1<p\leq\infty$ and let $\ups\in\ell^{p}(\ZZ)$. 
Then 
\begin{equation}\label{eq:p-norm}
		\sigma(H_V)\subset
	\bigg\{k+\frac{1}{k} \;\Big|\;\; k\in\CC\setminus\{0\}, \; |k|\leq 1 
	\,\mbox{ and }\, \bigg|k-\frac{1}{k}\bigg|\bigg(\frac{1-|k|^q}{1+|k|^q}\bigg)^{\!1/q} \leq
	\|\ups\|_{\ell^p(\ZZ)}\bigg\}.
\end{equation}
\end{Theorem}
Similarly to the continuous setting, see e.g.~\cite{Frank-TAMS-18, Cuenin-JFA-17}, complex	interpolation applied to an appropriate analytic family of Birman--Schwinger type operators yields the following alternative spectral enclosure for the case of $\ell^{p}$--potentials.
\begin{Theorem}\label{thm:p-norm.2}
Let $1<p\leq\infty$ and let $\ups\in\ell^{p}(\ZZ)$. 
Then 
\begin{equation}\label{eq:p-norm.2}
\sigma(H_V)\subset	\Big\{\la\in\CC\;\big|\;\; |\la^2-4|\dist(\la, [-2,2])^{2p-2} \leq \|\ups\|^{2p}_{\ell^p(\ZZ)}  \Big\}.
\end{equation}
\end{Theorem}
\begin{Remark}
 For $p=\infty$, the inequality in~\eqref{eq:p-norm.2} has to be understood as 
 \[
  \dist(\la,[-2,2])\leq\|\ups\|_{\ell^{\infty}(\ZZ)}.
 \]
 Note also that, if one sets $p=1$ in~\eqref{eq:p-norm.2}, one arrives at the bound~\eqref{eq:1-norm} without the exclusion of the interval~$(-2,2)$, however.
\end{Remark}

Notice the comparatively simpler form of the spectral enclosure from Theorem~\ref{thm:1-norm} in comparison with the ones of Theorems~\ref{thm:p-norm} and \ref{thm:p-norm.2}. The spectral enclosures from~\eqref{eq:p-norm} and \eqref{eq:p-norm.2} are compact sets which are symmetric with respect to both the real and the imaginary axes. However, in contrast to~\eqref{eq:1-norm}, they are always connected sets containing the entire interval $[-2,2]$. In general, none of the spectral enclosures from~\eqref{eq:p-norm} and \eqref{eq:p-norm.2} is a subset of the other. Naturally, taking their intersection yields the best result. Several plots of boundary curves of these sets as well as their comparison are presented in Section~\ref{sec:plots}.

The methodology used to deduce spectral enclosures of Theorems~\ref{thm:1-norm} and~\ref{thm:p-norm} is by no means new. It relies on the conventional Birman--Schwinger principle similarly as the analogous results for the (continuous) Schr{\" o}dinger and Dirac operators, see \eg~\cite{Frank-BLMS-11,Fan-Kre-Veg-JST18,Cue-Lap-Tre-14,Fan-Kre-LMP19,Fra-Sim-JST17} and also~\cite{Cue-Tre-JMAA16,Ibr-Kre-Lap}. The lastly mentioned papers served as a motivation for the current work. In comparison with the (continuous) Schr{\" o}dinger operators, where the spectral enclosure is a disk centered at the origin, the spectral enclosures from Theorems~\ref{thm:1-norm}, \ref{thm:p-norm} and \ref{thm:p-norm.2} have more interesting geometry. On the other hand, to estimate the norm of the Birman--Schwinger operator is technically less demanding in the discrete setting. Namely, it is done by elementary means in the case of $\ell^{1}$--potentials, it makes use of either the Schur test or discrete Young's inequality to deduce~\eqref{eq:p-norm}, or it employs a very particular form of Stein's complex interpolation to obtain~\eqref{eq:p-norm.2}.

The outline of the paper is as follows. In Section 2, the Birman--Schwinger principle is briefly recalled. Section~\ref{sec:proofs} contains proofs of Theorems~\ref{thm:1-norm}, \ref{thm:p-norm}, and~\ref{thm:p-norm.2} with two alternative proofs for Theorem~\ref{thm:p-norm}. In Section~\ref{Sect.Ex}, we discuss an example of the operator $H_{V}$ with a delta potential which demonstrates the optimality of the spectral enclosure~\eqref{eq:1-norm} for $\ell^1$--potentials. Final Section~\ref{sec:plots} is devoted to numerical illustrations of the results of Theorems~\ref{thm:1-norm}, \ref{thm:p-norm}, and~\ref{thm:p-norm.2}.

%=============================================================%
\section{The Birman--Schwinger principle}
%=============================================================%
%
The central role in our analysis is played by the 
Birman--Schwinger operator
\[
	K(\la) := |V|^{1/2} \, (H_0-\la)^{-1} \, V_{1/2}
\]
where $|V|^{1/2}e_{n}=\sqrt{|\ups_{n}|}e_{n}$ and 
$V_{1/2}e_{n}=\sgn(\ups_n)\sqrt{|\ups_{n}|}e_{n}$ with the 
complex signum function $\sgn\colon\CC\to\CC$ defined by
\[
	\sgn{z}=
	\begin{cases}
	\displaystyle
	\frac{z}{|z|} & \quad \mbox{ if } z\neq0,\\
	\;0 & \quad \mbox{ if } z=0.
	\end{cases}
\]
For $\la\in\CC\setminus[-2,2]$, the Birman--Schwinger operator 
is a bounded operator in $\ell^2(\ZZ)$ and its matrix 
elements read  
\begin{equation}\label{mat.elem.BS.op}
	K(\la)_{m,n} = \sqrt{|\ups_m|}\,(H_0-\la)^{-1}_{m,n} 
		\,\sqrt{|\ups_n|}\sgn(\ups_n), \quad \forall m,n\in\ZZ.
\end{equation}
Here the middle term corresponds to the matrix element 
of the resolvent of the discrete Laplace operator~$H_{0}$ 
which is given by
\begin{equation}\label{eq:resolvent_H_0}
	(H_0-\la)^{-1}_{m,n}=\left(H_{0}-k-k^{-1}\right)^{-1}_{m,n}
	=\frac{k^{|m-n|}}{k-k^{-1}}, \quad \forall m,n\in\ZZ,
\end{equation}
where $0<|k|<1$ and $\la=k+k^{-1}$, see~\cite[Chp.~1]{teschl00}. 
Here, it is often convenient to relate the spectral parameter $\la$ 
with $k$ by the Joukowsky transform $\lambda=k+k^{-1}$ which is 
a bijection between the punctured unit disk $0<|k|<1$ and the resolvent set 
$\CC\setminus[-2,2]$ of~$H_{0}$.

For $\la\in\CC\setminus[-2,2]$ and a bounded potential $V$, the  Birman--Schwinger principle is the equivalence
\begin{equation}\label{BS.disc.EV}
	\la\in\sigma(H_V) \quad 
	\Longleftrightarrow 
	\quad -1\in\sigma(K(\la))
\end{equation}
can be easily justified by usual arguments.  In particular, it follows that the implication
\begin{equation}
 \|K(\la)\|<1 \quad\Longrightarrow\quad \lambda\notin\sigma(H_V)
 \label{eq:BS.impl}
\end{equation}
holds true for $\lambda\notin[-2,2]$. Moreover, the implication $\Longrightarrow$ in~\eqref{BS.disc.EV} remains valid even if the spectra are replaced by point spectra. Indeed, if $H_{V}\psi=\lambda\psi$ for some $\lambda\in\CC\setminus[-2,2]$ and $\psi\in\ell^{2}(\ZZ)$, then simple manipulations with the eigenvalue equation yields $K(\la)\phi=-\phi$ for $\phi:=|V|^{1/2}\psi\in\ell^{2}(\ZZ)$ since $|V|^{1/2}$ is bounded by our assumption. Consequently, the Birman--Schwinger principle implies the inclusion
\begin{equation}\label{eq:point.spec.incl.BS}
 \sigma_{p}(H_V)\subset\left\{\la\in\CC\setminus[-2,2] \mid \|K(\lambda)\|\geq1\right\} \cup [-2,2].
\end{equation}

%=============================================================%
\section{Proofs}\label{sec:proofs}
%=============================================================%
%
\subsection{Proof of Theorem~\ref{thm:1-norm}}
The proof uses~\eqref{eq:point.spec.incl.BS}. Since the set on the right-hand side of~\eqref{eq:1-norm} always contains points~$\pm2$ we can distinguish the following two cases.

\smallskip
\noindent
\textbf{Case} $\la\notin[-2,2]$.
Let $k\in\CC$ be such that $0<|k|<1$ and $\la=k+k^{-1}$. 
It follows readily from~\eqref{eq:resolvent_H_0} that
\begin{equation}\label{estim.Green.fn}
	\left|(H_{0}-\la)^{-1}_{m,n}\right|<\frac{1}{\left|k-k^{-1}\right|}=\frac{1}{\sqrt{|\la^2-4|}}, \quad \forall m,n\in\ZZ.
\end{equation}
For $\ups\in\ell^{1}(\ZZ)$, the Birman--Schwinger operator $K(\la)$ is Hilbert--Schmidt (even trace class). Using \eqref{estim.Green.fn} and the Cauchy--Schwarz inequality, we can estimate 
\begin{align*}
 \|K(\lambda)\psi\|^{2}_{\ell^{2}(\ZZ)}&\leq\sum_{m\in\ZZ}\left(\sum_{n\in\ZZ}\sqrt{|\ups_{m}|}\left|(H_{0}-\la)^{-1}_{m,n}\right|\sqrt{|\ups_{n}|}|\psi_{n}|\right)^{\!2}\\
 &\leq\frac{\|\ups\|_{\ell^{1}(\ZZ)}}{|\lambda^{2}-4|}\left(\sum_{m\in\ZZ}\sqrt{|\ups_{n}|}|\psi_{n}|\right)^{\!2}\leq \frac{\|\ups\|^{2}_{\ell^{1}(\ZZ)}}{|\lambda^{2}-4|}\|\psi\|^{2}_{\ell^{2}(\ZZ)},
\end{align*}
for any~$\psi\in\ell^{2}(\ZZ)$. Hence
\begin{equation}\label{op.nor.bd.p=1}
	\|K(\la)\|\leq \frac{\|\ups\|_{\ell^{1}(\ZZ)}}{\sqrt{|\la^2-4|}}
\end{equation}
and, according to~\eqref{eq:point.spec.incl.BS}, if $\lambda\in\sigma_{p}(H_V)$, then 
\begin{equation}\label{EV.bd.prf}
	|\la^2-4|\leq \|\ups\|^2_{\ell^{1}(\ZZ)}.
\end{equation}

\smallskip
\noindent
\textbf{Case} $\la\in(-2,2)$. We make use of the fact that assumption $\ups\in\ell^{1}(\ZZ)$ actually implies $(-2,2)\cap\sigma_{p}(H_V)=\emptyset$. This seems to be a commonly known result but we did not find an exact reference for this claim concerning doubly-infinite Jacobi matrices with complex entries. Nevertheless, it follows readily from the existence of a particular solution $\phi(k)=\phi_{n}(k)$ of the three-term recurrence
\begin{equation}
 \phi_{n+1}+\left(\ups_{n}-k-k^{-1}\right)\phi_{n}+\phi_{n-1}=0
\label{eq:recur_evl_eq}
\end{equation}
which satisfies
\begin{equation}
 \lim_{n\to\infty}k^{-n}\phi_{n}(k)=1, \quad \forall k\in\CC,\ |k|\leq1, \; k\neq\pm1,
\label{eq:lim_jost}
\end{equation}
provided that $\ups\in\ell^{1}(\ZZ)$. The solution $\phi(k)$ is referred to as the Jost solution and its existence is proved, for example, in~\cite[Sec.~13.6]{Simon-OPUC2} for real valued~$\ups$. However, the reality of the potential is of no importance for the proof and hence the statement can be extended to complex valued~$\ups$ as well.

If $\lambda\in(-2,2)$ and $k+k^{-1}=\lambda$, then $|k|=1$ but $k\neq\pm1$. Consequently, the Jost solution~$\phi(k)$ exists. Moreover, the second linearly independent solution of~\eqref{eq:recur_evl_eq} can be chosen as $\phi(\bar{k})$
since the Wronskian
\[
 W(\phi(k),\phi(\bar{k}))=\lim_{n\to\infty}\bigl(\phi_{n}(k)\phi_{n+1}(\bar{k})-\phi_{n+1}(k)\phi_{n}(\bar{k})\bigr)=2\ii\Im k\neq0,
\]
where we have used~\eqref{eq:lim_jost} and the fact that $\bar{k}=k^{-1}$. Consequently, any solution~$\psi$ of the eigenvalue equation $H_V\psi=\lambda\psi$, for $\lambda\in(-2,2)$, is a linear combination of $\phi(k)$ and $\phi(\bar{k})$ where $\la=k+k^{-1}$ and $|k|=1$. Such~$\psi$ is square summable only if trivial. Thus~$\lambda$ cannot be an eigenvalue of~$H_V$. This completes the proof.
\qed
\subsection{Proof of Theorem~\ref{thm:p-norm}}
Let $k\in\CC$ be such that $0<|k|<1$ and $\la=k+k^{-1}$. 
It follows readily from~\eqref{eq:resolvent_H_0} that
\begin{equation}\label{res.estim.H0}
	|(H_0-\la)^{-1}_{m,n}|\leq \frac{|k|^{|m-n|}}{|k-k^{-1}|}, 
	\quad \forall m,n\in\ZZ.
\end{equation}

There are at least two proofs of Theorem~\ref{thm:p-norm}. The first one
makes use of a slight modification of the classical Schur test adjusted to our needs. The classical Schur test can be found in~\cite[Exer.~9, Chp.~II]{Conway-90b}.
\begin{Lemma}[Schur test]
 Let $K$ be a bounded operator on~$\ell^{2}(\ZZ)$ with matrix entries with respect to the standard basis denoted by $K_{m,n}$ for $m,n\in\ZZ$. Suppose that there is a set $\mathcal{Z}\subset\ZZ$ such that $K_{m,n}=0$ whenever $m\in\mathcal{Z}$ or $n\in\mathcal{Z}$. Then, for arbitrary weights $p_{j}>0$, $j\in\mathcal{I}:=\ZZ\setminus\mathcal{Z}$, one has
 \[
  \|K\| \leq 
	\bigg(\sup_{m\in I}\frac{1}{p_m}\sum_{n\in\mathcal{I}}p_n|K_{m,n}|\bigg)^{1/2}
	\bigg(\sup_{n\in I}\frac{1}{p_n}\sum_{m\in\mathcal{I}}p_m|K_{m,n}|\bigg)^{1/2}.
 \]
\end{Lemma}

The second proof applies discrete Young's inequality which can be proved by a simple modification of the proof of classical Young's inequality~\cite[Thm.~4.2]{Lie-Los01}. 

\begin{Theorem}[discrete Young's inequality]
 Let $p,q,r \geq 1$ be such that  
\[
 \frac{1}{p}+\frac{1}{q}+\frac{1}{r}=2.
\]
Then, for any $f\in\ell^{p}(\ZZ)$, $g\in\ell^{q}(\ZZ)$, and $h\in\ell^{r}(\ZZ)$, one has
\begin{equation}
 \sum_{i,j\in\ZZ}|f_{i}||g_{j-i}||h_{j}|\leq \|f\|_{p}\|g\|_{q}\|h\|_{r}.
 \label{eq:young}
\end{equation}
\end{Theorem}

\begin{proof}[The first proof of Theorem~\ref{thm:p-norm}]
We apply the Schur test to the Birman--Schwinger 
operator $K(\lambda)$ defined in~\eqref{mat.elem.BS.op}. Since $K_{m,n}(\lambda)=0$ whenever $\ups_{n}=0$ or $\ups_{n}=0$, we put 
$\mathcal{I}:=\{j\in\ZZ \mid \ups_j\neq0\}$ and the positive weights are chosen as $p_j:=\sqrt{|\ups_j|}$ for $j\in\mathcal{I}$. Then, taking the definition~\eqref{mat.elem.BS.op} and~\eqref{res.estim.H0} into account, we obtain the estimate 
\begin{equation}\label{Estim.BS.op.lp}
	\|K(\la)\| \leq 
	\frac{1}{|k-k^{-1}|} \sup_{m\in\mathcal{I}}\sum_{n\in\ZZ} |\ups_n| |k|^{|m-n|}.
\end{equation}
For the case $p=\infty$, we infer from \eqref{Estim.BS.op.lp} that
\begin{equation}\label{Estim.BS.p=infty}
	\|K(\la)\| \leq 
	\frac{\|\ups\|_{\ell^{\infty}(\ZZ)}}{|k-k^{-1}|} \sum_{n\in\ZZ} |k|^{|n|}
	=\frac{\|\ups\|_{\ell^{\infty}(\ZZ)}}{|k-k^{-1}|}\frac{1+|k|}{1-|k|}.
\end{equation}
Now consider the case $p\in(1,\infty)$. 
For every fixed $m\in I$, H\"older's inequality yields
\begin{equation}
	\sum_{n\in\ZZ} |\ups_n| |k|^{|m-n|} 
	\leq 
	\|\ups\|_{\ell^p(\ZZ)}	\bigg(\sum_{n\in\ZZ} |k|^{|n|q} \bigg)^{1/q}
	=\|\ups\|_{\ell^p(\ZZ)}\bigg(\frac{1+|k|^q}{1-|k|^q}\bigg)^{1/q}
\end{equation}
and thus
\begin{equation}\label{estim.BS.p-norm}
	\|K(\la)\|\leq\frac{\|\ups\|_{\ell^p(\ZZ)}}{|k-k^{-1}|}\bigg(\frac{1+|k|^q}{1-|k|^q}\bigg)^{1/q}.
\end{equation}
In view of the estimates \eqref{Estim.BS.p=infty} and \eqref{estim.BS.p-norm}, 
the Birman--Schwinger principle \eqref{BS.disc.EV} thus implies that 
$\la=k+k^{-1}\in\CC\setminus[-2,2]$ 
cannot belong to the spectrum of $H_V$ unless it holds that 
\begin{equation}\label{EV.bd.lp.prf}
	\|\ups\|_{\ell^p(\ZZ)}
	\geq \left|k-\frac{1}{k}\right|\bigg(\frac{1-|k|^q}{1+|k|^q}\bigg)^{1/q}.
\end{equation}
Finally, noticing that also the interval $[-2,2]$ is always included in the set on the right-hand side of~\eqref{eq:p-norm}, corresponding to the case when $|k|=1$, the proof is completed.
\end{proof}

\begin{proof}[The second proof of Theorem~\ref{thm:p-norm}]
Suppose $\ups\in\ell^{p}(\ZZ)$ for $1<p\leq\infty$ and $\lambda\in\CC\setminus[-2,2]$ and $k\in\CC$, $0<|k|<1$, such that $\lambda=k+k^{-1}$. For $\phi,\psi\in\ell^{2}(\ZZ)$, we have
\begin{equation}
 \left|\langle\phi,K(\lambda)\psi\rangle\right|\leq\frac{1}{\left|k-k^{-1}\right|}\sum_{m,n\in\ZZ}|f_{m}||g_{n-m}||h_{m}|,
 \label{eq:BS_sesq}
\end{equation}
where
\[
 f_{n}:=\sqrt{|\ups_{n}|}\phi_{n}, \quad g_{n}:=|k|^{|n|}, \quad h_{n}:=\sqrt{|\ups_{n}|}\psi_{n}, \quad n\in\ZZ.
\]
Notice that, by H{\" o}lder's inequality, we have
\begin{equation}
 \|f\|_{\ell^{2p/(p+1)}(\ZZ)}\leq\sqrt{\|\ups\|_{\ell^{p}(\ZZ)}}\|\phi\|_{\ell^{2}(\ZZ)}
 \quad\mbox{ and }\quad
 \|h\|_{\ell^{2p/(p+1)}(\ZZ)}\leq\sqrt{\|\ups\|_{\ell^{p}(\ZZ)}}\|\psi\|_{\ell^{2}(\ZZ)}.
\label{eq:f_h_norm_estim}
\end{equation}
Moreover, for any $q\geq1$, it is clearly true that
\begin{equation}
 \|g\|_{\ell^{q}(\ZZ)}=1+2\left(\sum_{n=1}^{\infty}|k|^{qn}\right)^{\!1/q}=\left(\frac{1+|k|^{q}}{1-|k|^{q}}\right)^{\!1/q}.
\label{eq:g_norm}
\end{equation}
Thus, we can apply Young's inequality~\eqref{eq:young} in~\eqref{eq:BS_sesq} with the indices $p$ and $r$ replaced by $2p/(p+1)$ (or by $2$ if $p=\infty$) and $q$ is the H{\" o}lder dual index to~$p$ ($q=1$ if $p=\infty$). Taking also~\eqref{eq:f_h_norm_estim} and~\eqref{eq:g_norm} into account, we obtain the estimate
\[
 \left|\langle\phi,K(\lambda)\psi\rangle\right|\leq\frac{\|\ups\|_{\ell^{p}(\ZZ)}}{\left|k-k^{-1}\right|}\left(\frac{1+|k|^{q}}{1-|k|^{q}}\right)^{\!1/q}\|\phi\|_{\ell^{2}(\ZZ)}\|\psi\|_{\ell^{2}(\ZZ)},
\]
for all $\phi,\psi\in\ell^{2}(\ZZ)$. In other words, we arrived at~\eqref{estim.BS.p-norm} again and the same reasoning as in the previous proof implies the desired inclusion~\eqref{eq:BS.impl}.
\end{proof}
\subsection{Proof of Theorem~\ref{thm:p-norm.2}}
The proof makes use of Stein's complex interpolation theorem, see~\cite[Thm.~1]{Ste-TAMS-1956}. However, only a special version of the theorem is sufficient for our needs. Its statement is as follows.

\begin{Theorem}[Stein's interpolation]\label{thm:stein}
 Let $T_{z}:\ell^{2}(\ZZ)\to\ell^{2}(\ZZ)$ be a family of operators analytic in the strip $0<\Re z<1$ and continuous and uniformly bounded in its closure $0\leq\Re z\leq1$. Suppose further that there exist constants $A_{0}$ and $A_{1}$ such that
 \[
  \|T_{\ii y}\|\leq A_{0} \quad\mbox{ and }\quad \|T_{1+\ii y}\|\leq A_{1},
 \]
 for all $y\in\RR$. Then, for any $\theta\in[0,1]$, one has
 \[
  \|T_{\theta}\|\leq A_{0}^{1-\theta}A_{1}^{\theta}.
 \]
\end{Theorem}

\begin{proof}[Proof of Theorem~\ref{thm:p-norm.2}]
Let $\la\in\CC\setminus[-2,2]$ be fixed. If $p=\infty$, one readily estimates the Birman--Schwinger operator getting
\[
 \|K(\la)\|\leq\||V|^{1/2}|\|\|(H_{0}-\la)^{-1}\|\|V_{1/2}\|\leq\frac{\|\ups\|_{\infty}}{\dist(\la,[-2,2])}.
\]
Then~\eqref{eq:BS.impl} implies~\eqref{eq:p-norm.2} in the particular case $p=\infty$.

Suppose $1<p<\infty$. Define the operator family 
\[
T_{z}:=|V|^{z p/2}(H_0-\la)^{-1}|V|^{z p/2},
\]
for $z\in\CC$ with $0\leq\Re z\leq 1$. Note that $T_{z}$ is continuous in the closed strip $0\leq\Re z\leq 1$ and analytic in its interior. Moreover, estimating similarly as above, we get
\[
 \sup_{0\leq\Re z\leq 1}\|T_{z}\|\leq\frac{\max(1,\|\ups\|_{\infty}^{p})}{\dist(\la,[-2,2])}.
\]
Therefore, $T_{z}$ is uniformly bounded for $0\leq\Re z\leq 1$.

Further, since $\ups\in\ell^p(\ZZ)$ by the hypothesis, we have $|\ups|^{p}\in\ell^1(\ZZ)$. Consequently, we can apply~\eqref{op.nor.bd.p=1} to get
\[
	\|T_{1+\ii y}\|=\||V|^{p/2}(H_{0}-\la)^{-1}|V|^{p/2}\|\leq
	\frac{\|\ups\|^p_{\ell^p(\ZZ)}}{\sqrt{|\la^2-4|}},
\]
for any $y\in\RR$. Moreover, for all $y\in\RR$, we have also the trivial estimate
	\[
	\|T_{\ii y}\|\leq \frac{1}{\dist(\la,[-2,2])}.
	\]
Hence, we may apply Theorem~\ref{thm:stein} with $\theta=1/p$ which yields
\begin{equation*}
\|T_{1/p}\|=\|K(\la)\|\leq
\frac{\|\ups\|_{\ell^p(\ZZ)}}{|\la^2-4|^{\frac{1}{2p}}\dist(\la, [-2,2])^{1-\frac{1}{p}}}
\end{equation*}
and~\eqref{eq:BS.impl} completes the proof.
\end{proof}

%

%=============================================================%
\section{Optimality of the eigenvalue bound for $\ell^{1}$-potentials}\label{Sect.Ex}
%=============================================================%
\noindent
In this section, we demonstrate that the result~\eqref{eq:1-norm} is sharp in 
the sense that to any boundary point of the spectral enclosure which is not located in~$[-2,2]$, there exists an $\ell^{1}$--potential~$V$ so that this boundary point is an eigenvalue of $H_{V}$.
The excluded boundary points occur only if $\|\ups\|_{\ell^{1}(\ZZ)|}\leq2$ and they are $\pm\sqrt{4-\|\ups\|_{\ell^{1}(\ZZ)}^{2}}$.

For this purpose, we consider the operator $H_{V}$ with the potential $V$ determined by the sequence
\begin{equation}
\ups_{n}:=\omega\delta_{n,0}, \quad \forall n\in\ZZ,
\label{eq:delta_pot_q}
\end{equation}
where $\omega\in\CC$ is a coupling constant and $\delta_{m,n}$ stands for the Kronecker delta. The eigenvalue equation 
$H_{V}\psi=\lambda\psi$ reads
\[
\psi_{n+1}+\omega\delta_{n,0}\psi_{n}+\psi_{n-1}=\lambda\psi_{n}, \quad n\in\ZZ.
\]
It is convenient to write $\omega=k^{-1}-k$ for $0<|k|\leq1$. Such $k$ exists unique inside the unit disk $|k|<1$ if $\omega\in\CC\setminus[-2\ii,2\ii]$, there are exactly two on the unit circle $|k|=1$, $k\neq\pm\ii$ if $\omega\in(-2\ii,2\ii)$, and $k=\pm\ii$ if $\omega=\mp2\ii$.
In any case, one readily verifies that the vector
\[
\psi_{n}:=k^{|n|}, \quad n\in\ZZ
\]
fulfills the eigenvalue equation
\[
 H_{V}\psi=\left(k+k^{-1}\right)\psi.
\]
If, in addition, $\omega\in\CC\setminus[-2\ii,2\ii]$, then $\psi\in\ell^{2}(\ZZ)$ and hence
\[
 \lambda_{\omega}:=k+k^{-1}=\sqrt{4+\omega^{2}}\in\sigma_{\textrm{p}}(H_{V}).
\]
Obviously, $|\lambda_{\omega}^{2}-4|=|\omega|^{2}=\|q\|_{\ell^{1}(\ZZ)}^{2}$
and thus the eigenvalue
$\lambda_{\omega}$ is a boundary point of the spectral enclosure from~\eqref{eq:1-norm}.
Moreover, let us remark that, in this particular example, the Birman--Schwinger operator takes the simple form
\[
K_{m,n}(\la)=\begin{cases}
\frac{\omega}{k-k^{-1}} & \quad \mbox{ if } m=n=0,\\
0 & \quad \mbox{ otherwise,}
\end{cases}
\]
for $\lambda=k+k^{-1}$ with $0<|k|<1$. Thus the Birman--Schwinger principle \eqref{BS.disc.EV} together with the fact that $\sigma(H_{0})=[-2,2]$ shows
\[
\sigma(H_{V})=[-2,2]\cup\{\lambda_{\omega}\},
\]
where $\lambda_{\omega}\in[-2,2]$ if and only if $\omega\in[-2\ii,2\ii]$, in which case 
%$\lambda_{\omega}$ is not an eigenvalue of~$H_{V}$.
$\lambda_{\omega}\notin\sigma_{\rm{p}}(H_{V})$.

On the contrary, it is not difficult to see that to any boundary point~$z\notin[-2,2]$ of the spectral enclosure~\eqref{eq:1-norm} for $H_{V}$ with $\|\ups\|_{\ell^{1}}>0$, there exists $k\in\CC$ with $0<|k|<1$ and such that $z=k+k^{-1}$ and
$
 \left|k-k^{-1}\right|=\|\ups\|_{\ell^{1}}.
$
Hence, if we put $\omega:=k^{-1}-k$, the previous analysis shows that $z$ is an eigenvalue of~$H_{V}$ with the potential~$V$ given by the sequence~\eqref{eq:delta_pot_q}.

%=============================================================%
\section{Plots of the spectral enclosures from Theorems~\ref{thm:1-norm}, \ref{thm:p-norm}, and \ref{thm:p-norm.2}}\label{sec:plots}
%=============================================================%
\noindent First Figure~\ref{fig:indef} shows the spectral enclosures from Theorem~\ref{thm:1-norm}.
\begin{figure}[htb!]
	\centering
	\includegraphics[width= 0.8 \textwidth]{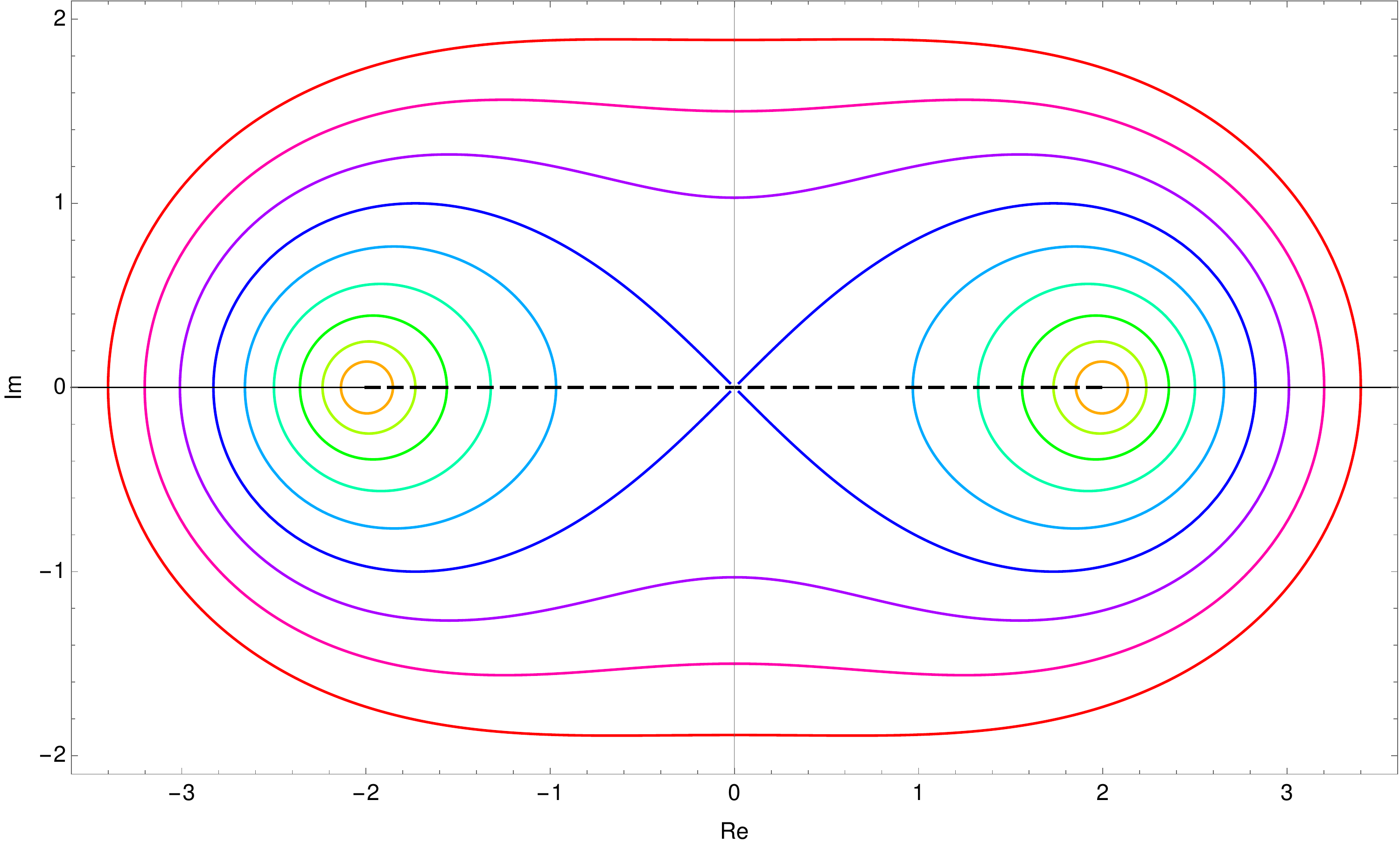}
	\caption{\small{The plots of the expanding boundary curves 
			corresponding to the spectral enclosure~\eqref{eq:1-norm} for $\|\ups\|_{\ell^{1}(\ZZ)}=j/4$, $j=3,\ldots,11$.}}
	\label{fig:indef}
\end{figure}

Second, we provide several plots as an illustration of the spectral enclosures from Theorem~\ref{thm:p-norm} in Figure~\ref{fig:sp_thm2}. Denoting by $Q:=\|v\|_{\ell^{p}(\ZZ)}$, the set in~\eqref{eq:p-norm} is determined by two parameters $Q\geq0$ and $q\geq1$. Without going into details, we remark that the boundary curve $\Gamma_{q,Q}$ given by the equation
\[
 \bigg|k-\frac{1}{k}\bigg|\bigg(\frac{1-|k|^q}{1+|k|^q}\bigg)^{1/q}=Q
\]
for $0<|k|<1$, can be parametrized in the first quadrant ($\Re z\ge0$ and $\Im z\geq0$) as
\begin{align*}
 \Re \Gamma_{q,Q}(t)&=2\cosh(t)\sqrt{\cosh^{2}(t)-\frac{Q^{2}}{4}\coth^{\frac{2}{q}}\!\left(\frac{qt}{2}\right)},\\
 \Im \Gamma_{q,Q}(t)&=2\sinh(t)\sqrt{-\sinh^{2}(t)+\frac{Q^{2}}{4}\coth^{\frac{2}{q}}\!\left(\frac{qt}{2}\right)}, 
\end{align*}
for $t_{\min}\leq t \leq t_{\max}$, where $t_{\min}$ and $t_{\max}$ are the unique solutions of the equations
\[\cosh^{2}(t)=\frac{Q^{2}}{4}\coth^{\frac{2}{q}}\!\left(\frac{qt}{2}\right), 
\quad\mbox{ and }\quad \sinh^{2}(t)=\frac{Q^{2}}{4}\coth^{\frac{2}{q}}\!\left(\frac{qt}{2}\right),
\]
for $t>0$, respectively. This parametrization is used in the plots in Figure~\ref{fig:sp_thm2}.

Third set of plots shows the spectral enclosures from Theorem~\ref{thm:p-norm.2}, see Figure~\ref{fig:sp_thm3}. Finally, we compare the spectral enclosures from Theorems~\ref{thm:p-norm} and~\ref{thm:p-norm.2} for six values of~$p$ in Figure~\ref{fig:compar}. The plots indicate that, for general $p>1$, no spectral enclosure is a subset of the other one.

\begin{figure}[H]
    \centering
    \begin{subfigure}[c]{0.49\textwidth}
	\includegraphics[width=\textwidth]{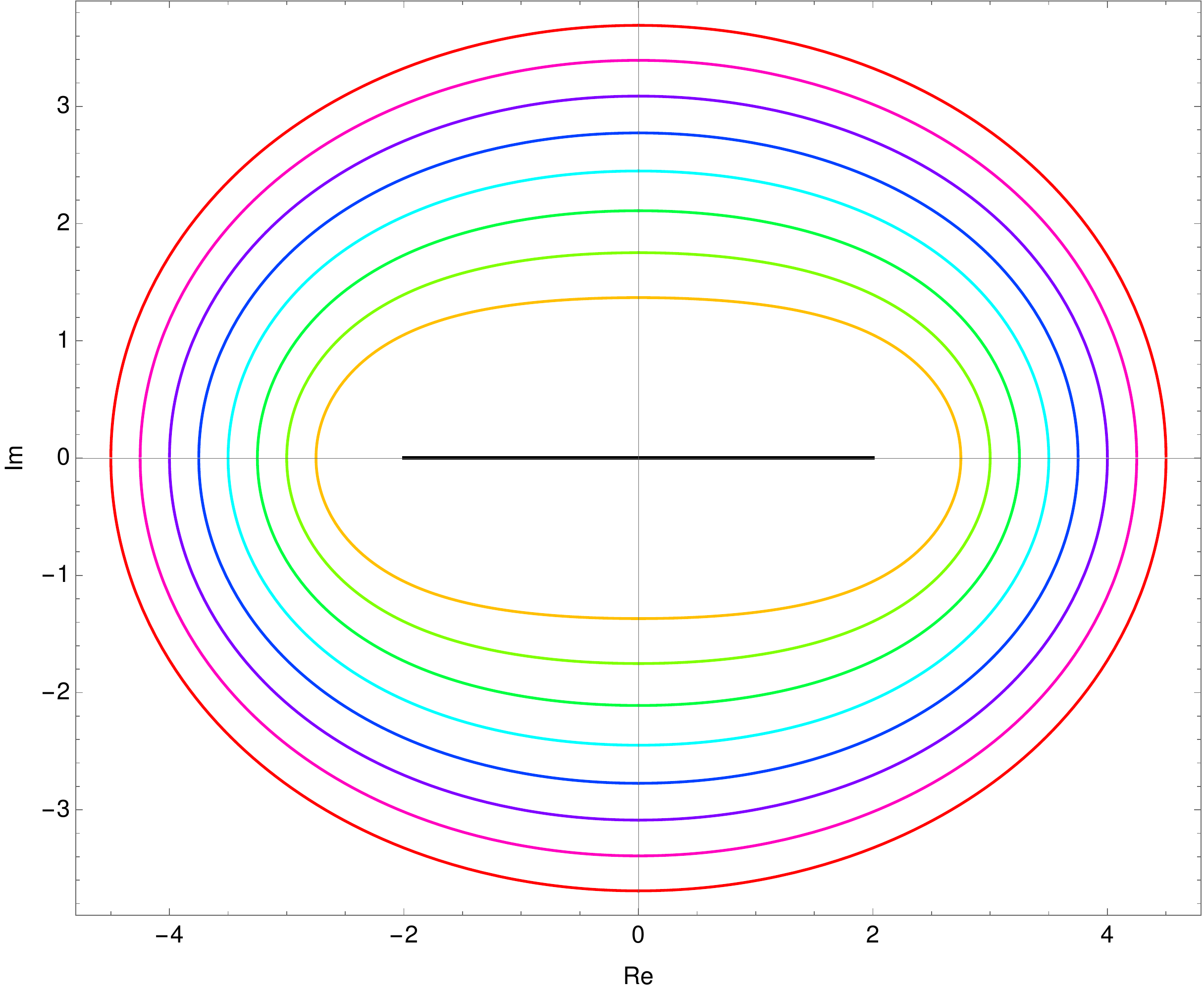}
    \caption{$q=1$}
    \end{subfigure}
    \begin{subfigure}[c]{0.49\textwidth}
	\includegraphics[width=\textwidth]{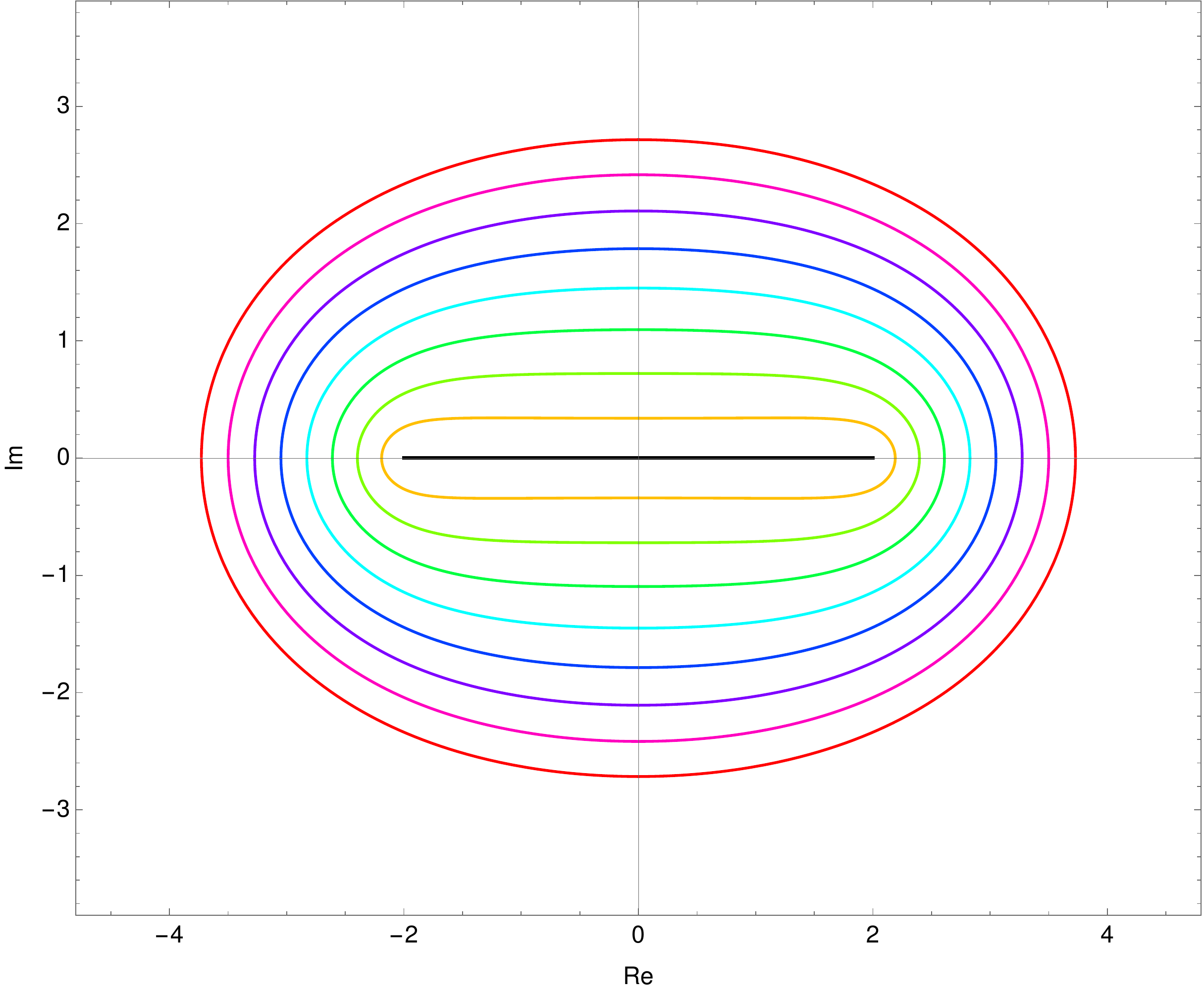}
	\caption{$q=9/8$}
    \end{subfigure}
    \vskip6pt
    \begin{subfigure}[c]{0.49\textwidth}
	\includegraphics[width=\textwidth]{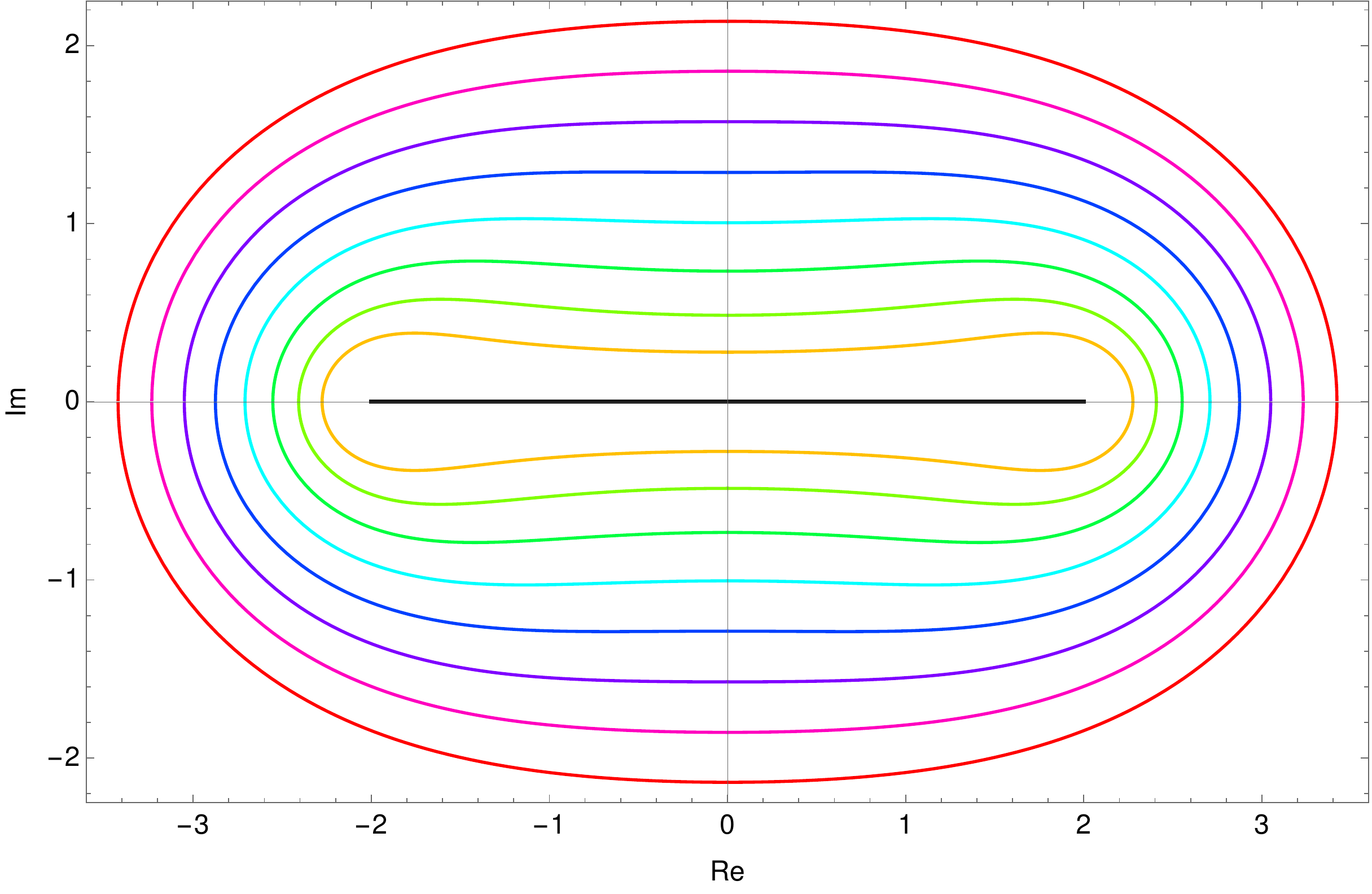}
	\caption{$q=2$}
    \end{subfigure}
    \begin{subfigure}[c]{0.49\textwidth}
	\includegraphics[width=\textwidth]{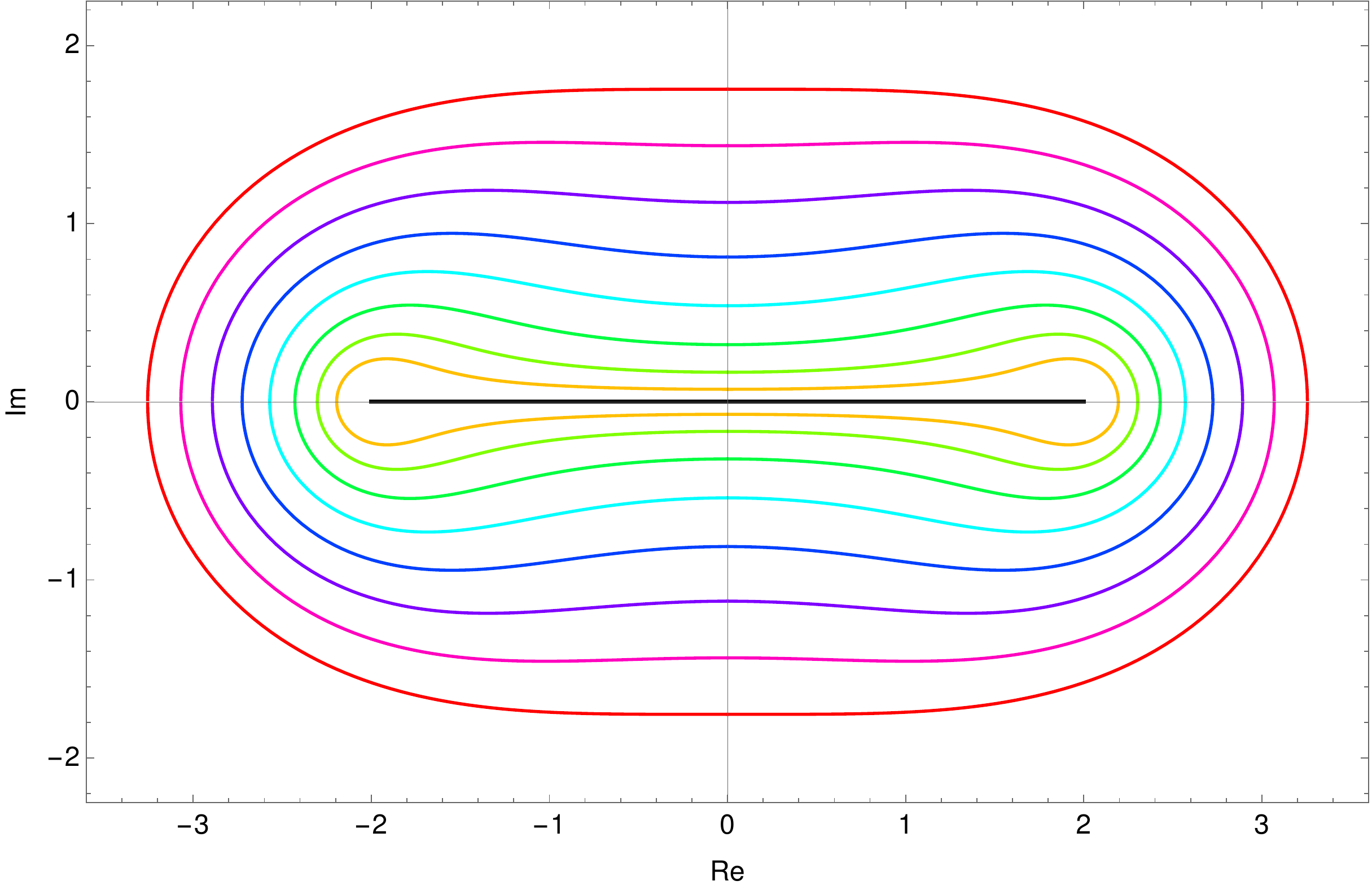}
	\caption{$q=3$}
    \end{subfigure}
    \vskip6pt
    \begin{subfigure}[c]{0.49\textwidth}
	\includegraphics[width=\textwidth]{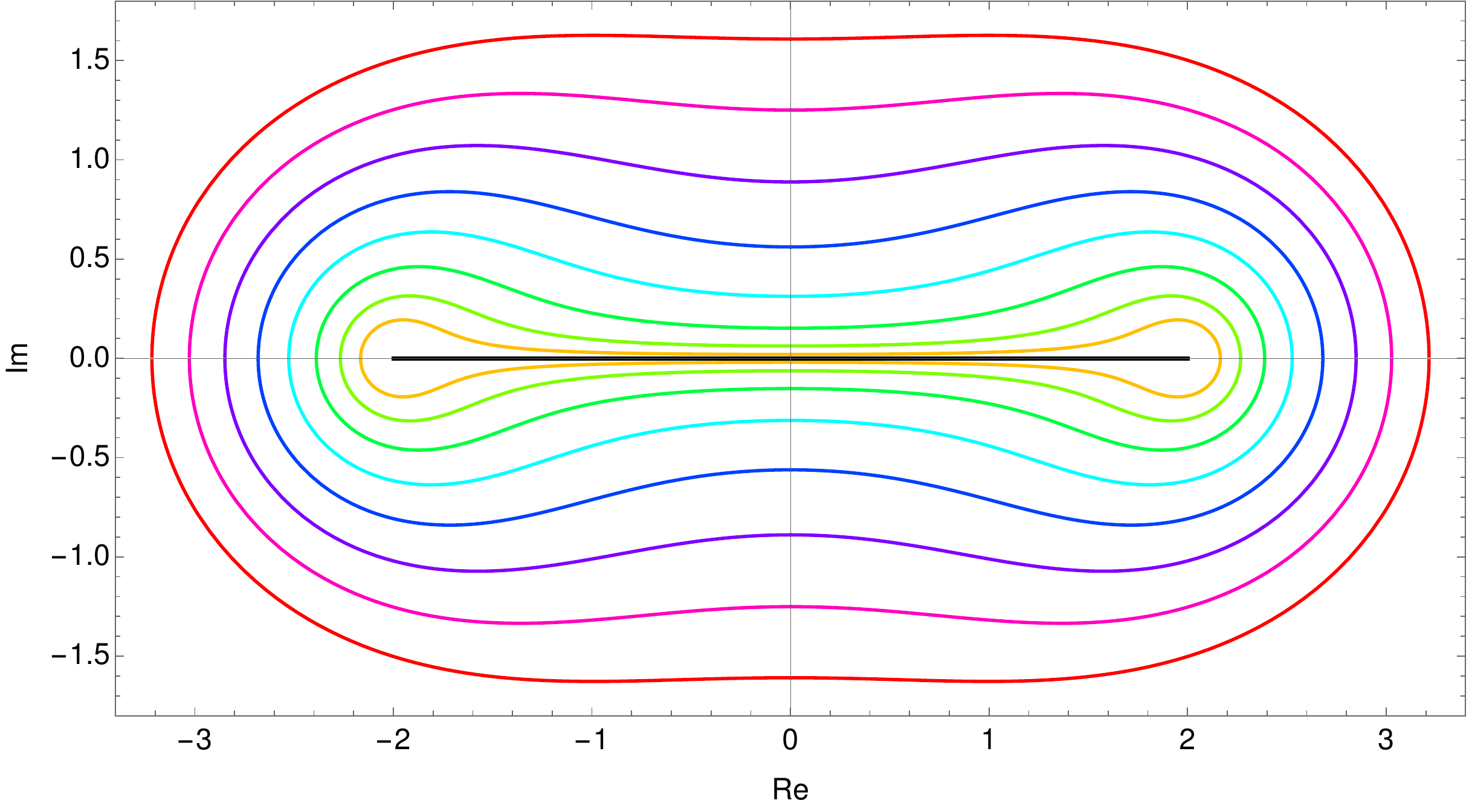}
	\caption{$q=4$}
    \end{subfigure}
    \begin{subfigure}[c]{0.49\textwidth}
	\includegraphics[width=\textwidth]{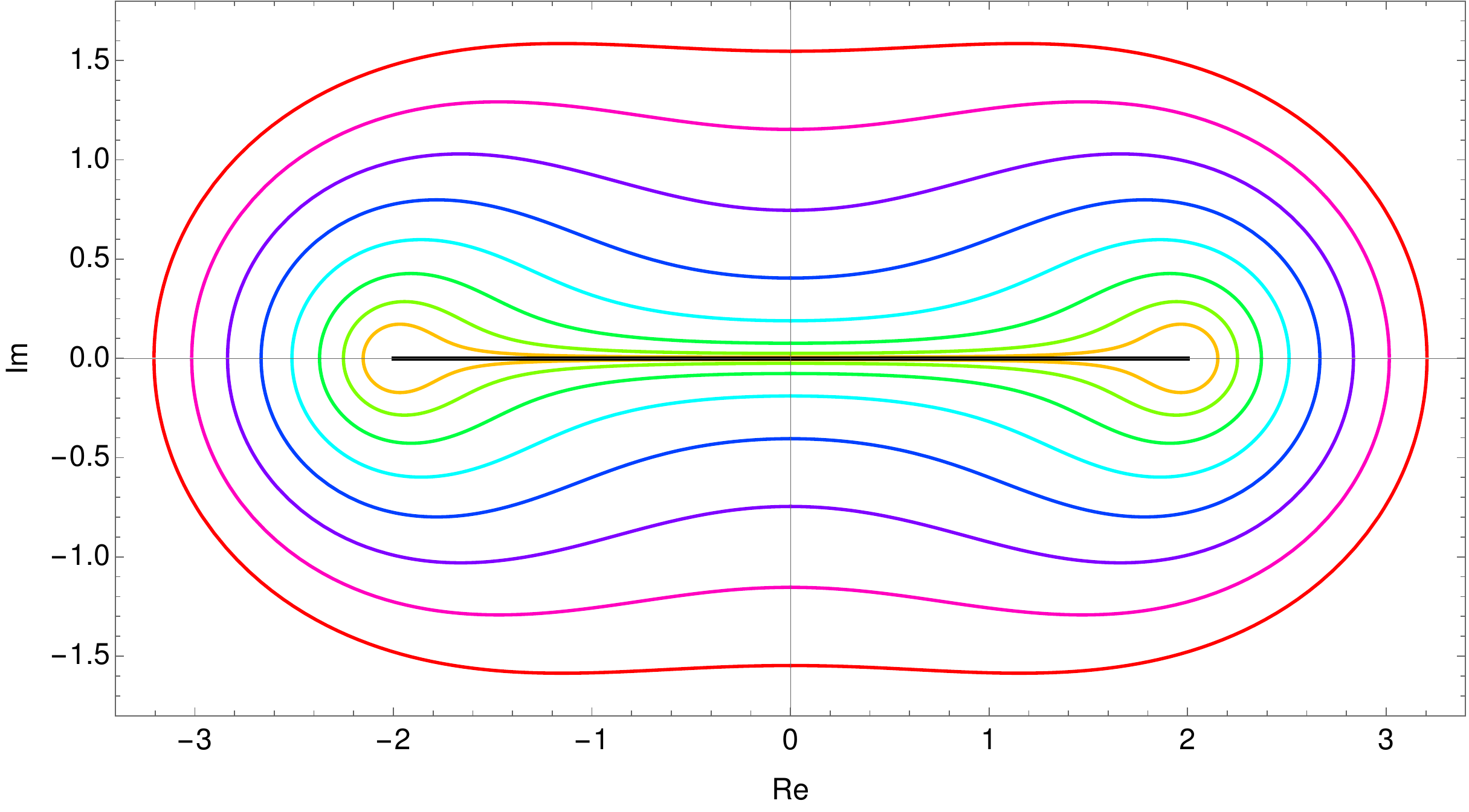}
	\caption{$q=5$}
    \end{subfigure}
    \caption{\small{The plots of the expanding boundary curves 
			corresponding to the spectral enclosure 
			\eqref{eq:p-norm} for six choices of the parameter~$q$ and $\|\ups\|_{\ell^{p}(\ZZ)}=j/4$, $j=3,\ldots,10$.}}
    \label{fig:sp_thm2}
\end{figure}

\begin{figure}[H]
    \centering
    \begin{subfigure}[c]{0.49\textwidth}
	\includegraphics[width=\textwidth]{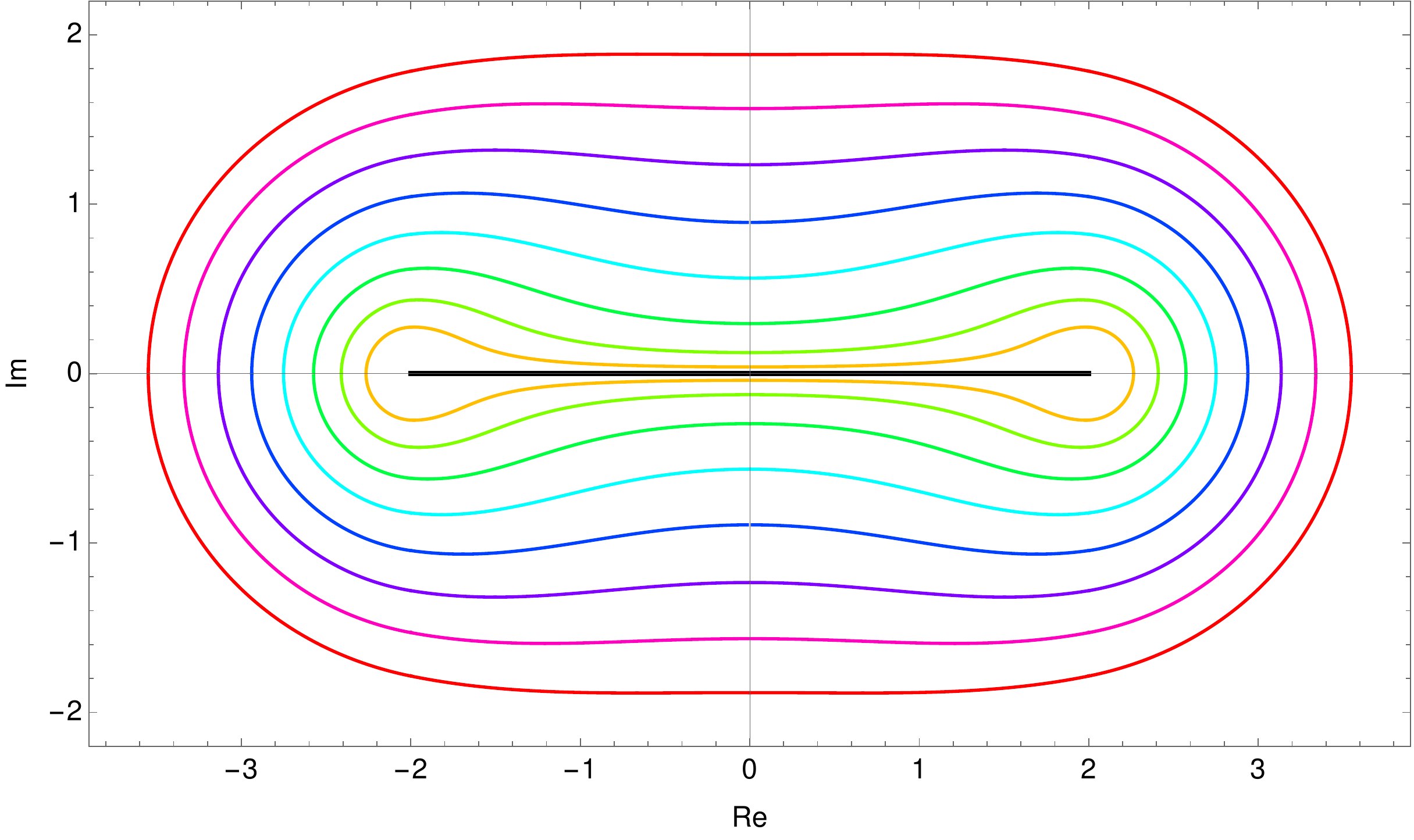}
    \caption{$p=4/3$}
    \end{subfigure}
    \begin{subfigure}[c]{0.49\textwidth}
	\includegraphics[width=\textwidth]{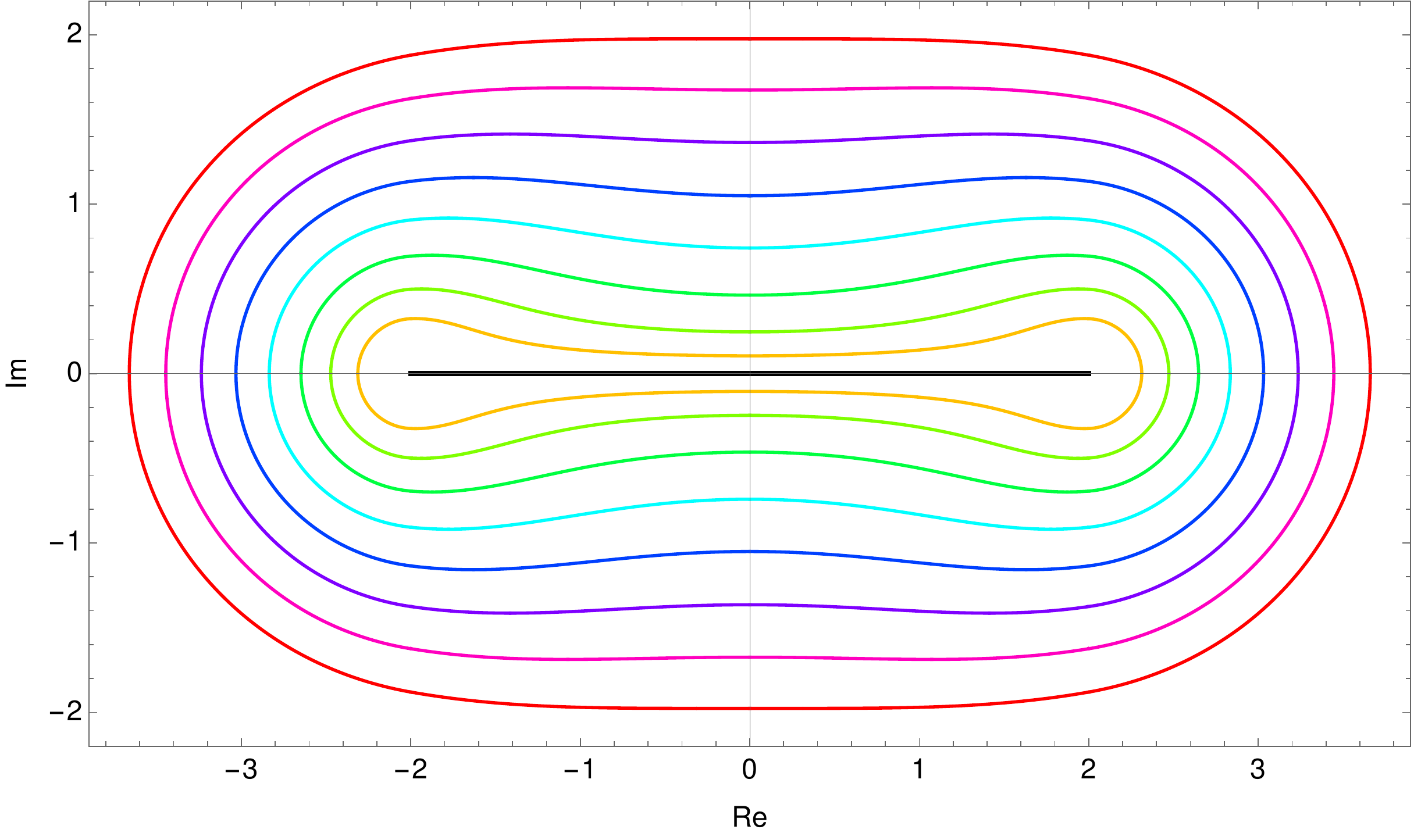}
	\caption{$p=3/2$}
    \end{subfigure}
    \vskip6pt
    \begin{subfigure}[c]{0.49\textwidth}
	\includegraphics[width=\textwidth]{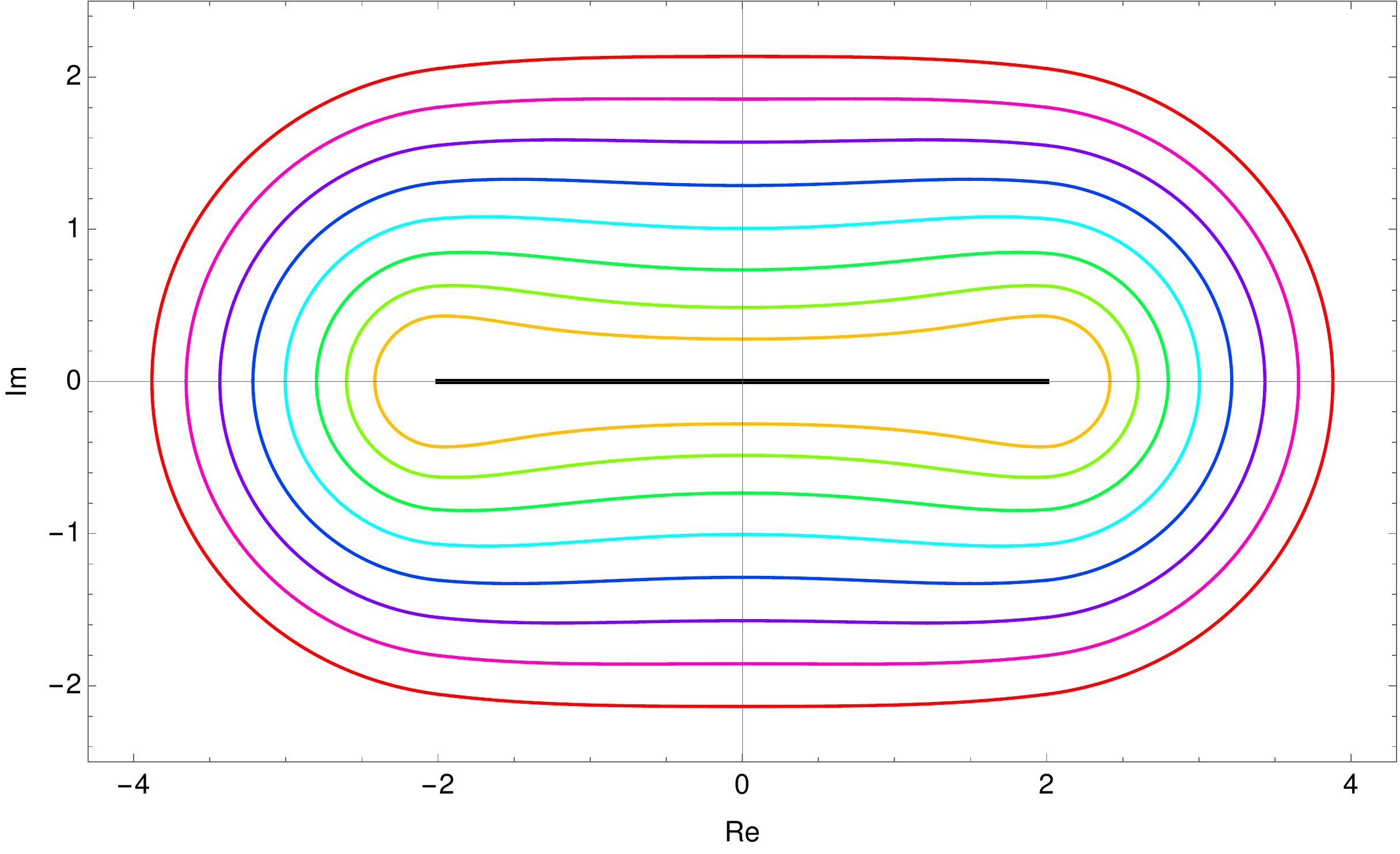}
	\caption{$p=2$}
    \end{subfigure}
    \begin{subfigure}[c]{0.49\textwidth}
	\includegraphics[width=\textwidth]{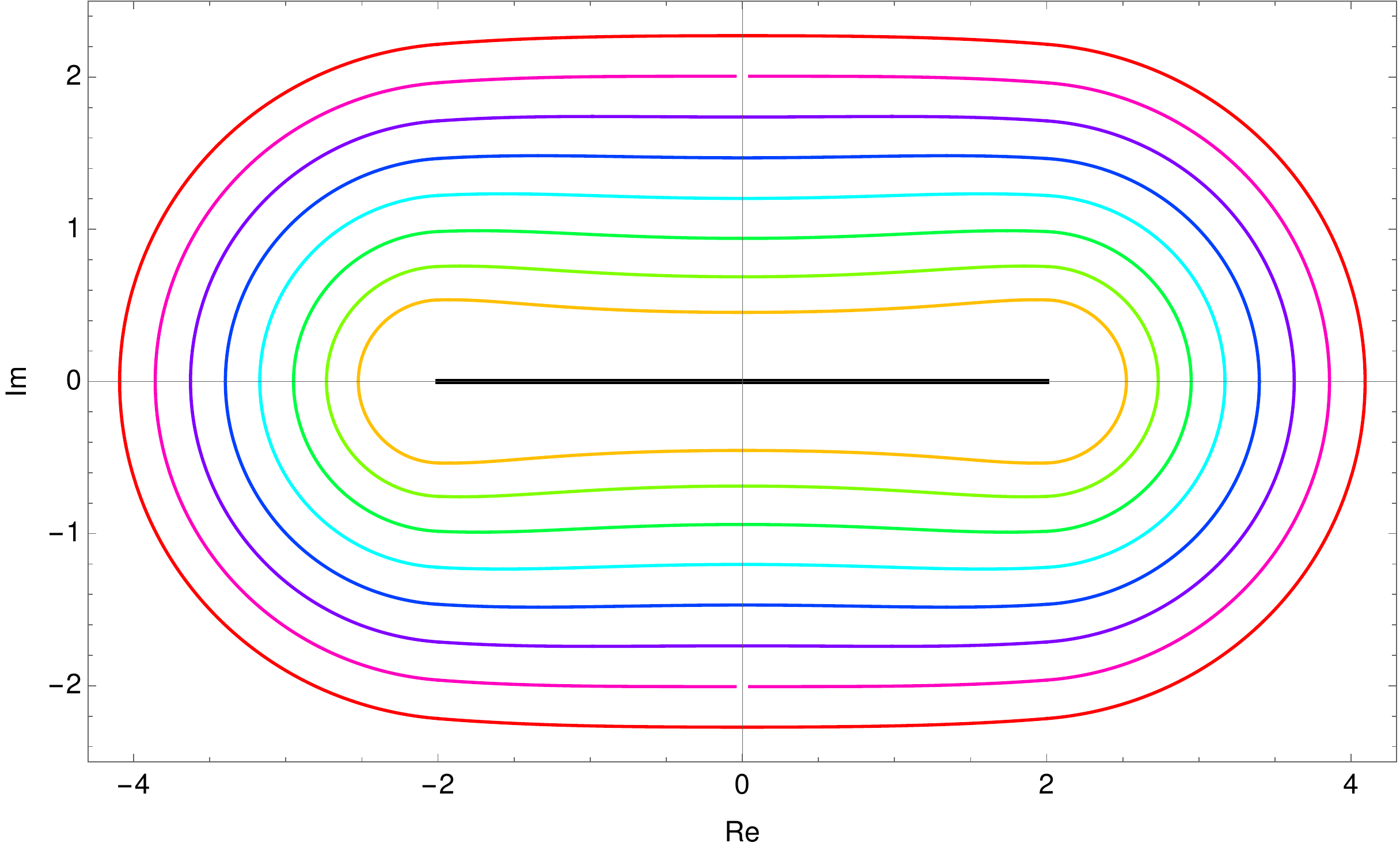}
	\caption{$p=3$}
    \end{subfigure}
    \vskip6pt
    \begin{subfigure}[c]{0.49\textwidth}
	\includegraphics[width=\textwidth]{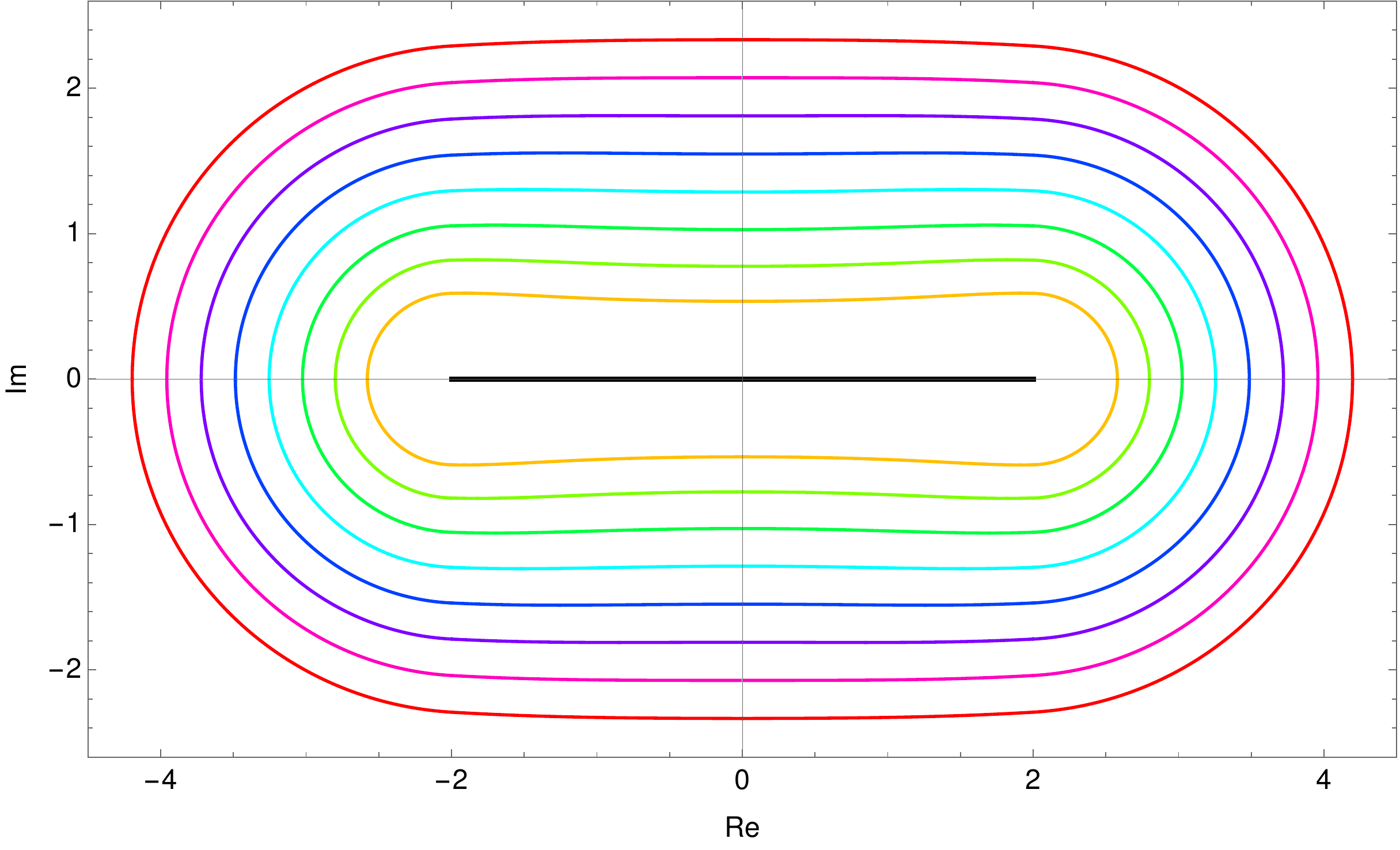}
	\caption{$p=4$}
    \end{subfigure}
    \begin{subfigure}[c]{0.49\textwidth}
	\includegraphics[width=\textwidth]{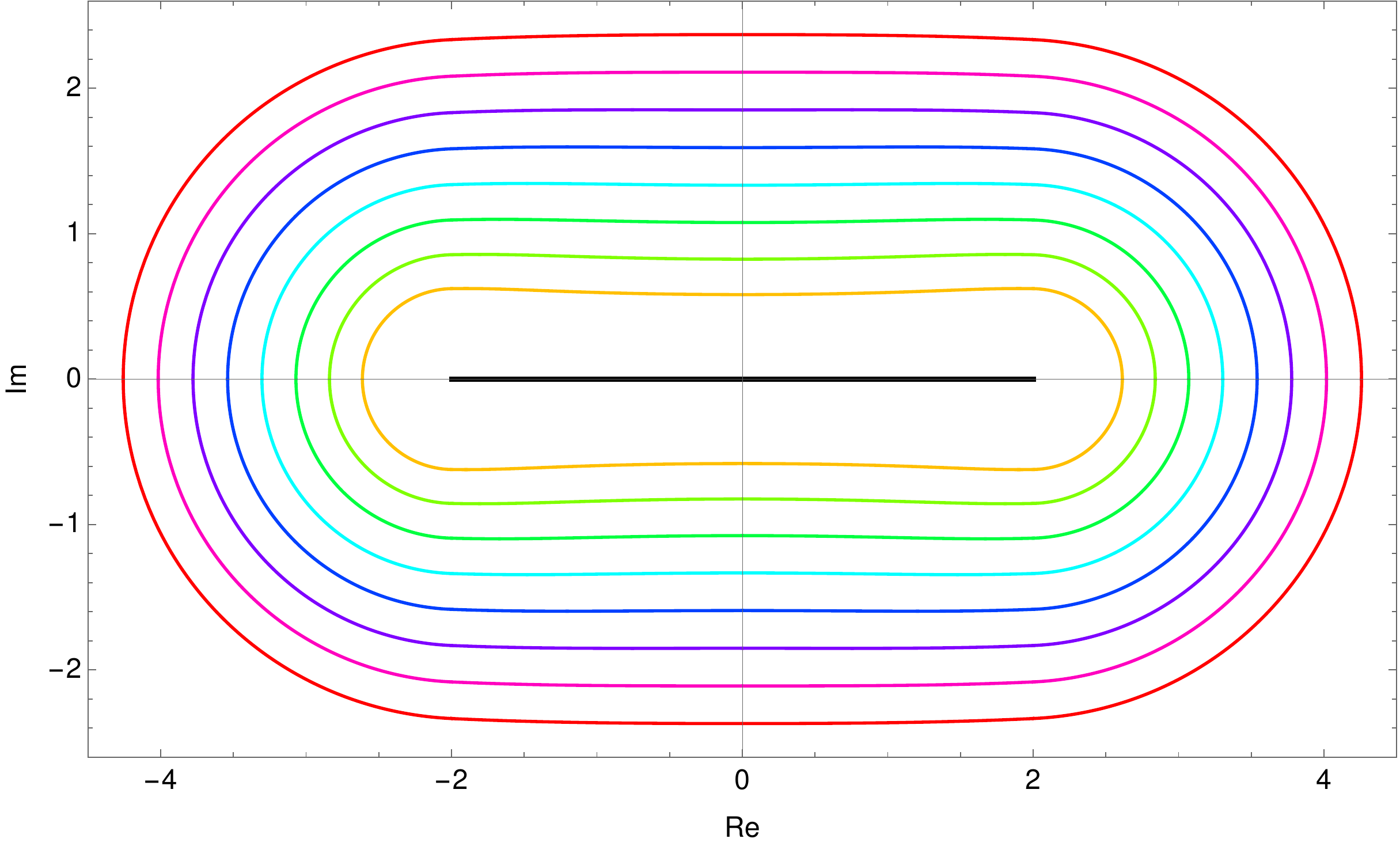}
	\caption{$p=5$}
    \end{subfigure}
    \caption{\small{The plots of the expanding boundary curves 
			corresponding to the spectral enclosure 
			\eqref{eq:p-norm.2} for six choices of the parameter~$p$ and $\|\ups\|_{\ell^{p}(\ZZ)}=j/4$, $j=3,\ldots,10$.}}
    \label{fig:sp_thm3}
\end{figure}

\begin{figure}[H]
    \centering
    \begin{subfigure}[c]{0.49\textwidth}
	\includegraphics[width=\textwidth]{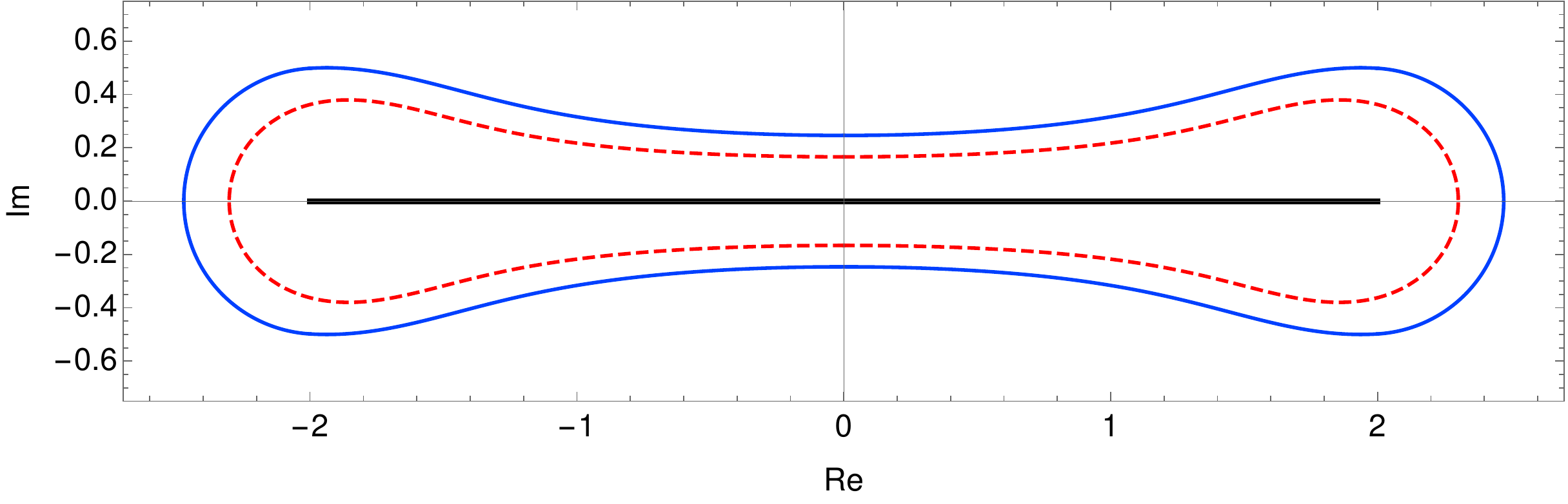}
    \caption{$p=3/2$}
    \end{subfigure}
    \begin{subfigure}[c]{0.49\textwidth}
	\includegraphics[width=\textwidth]{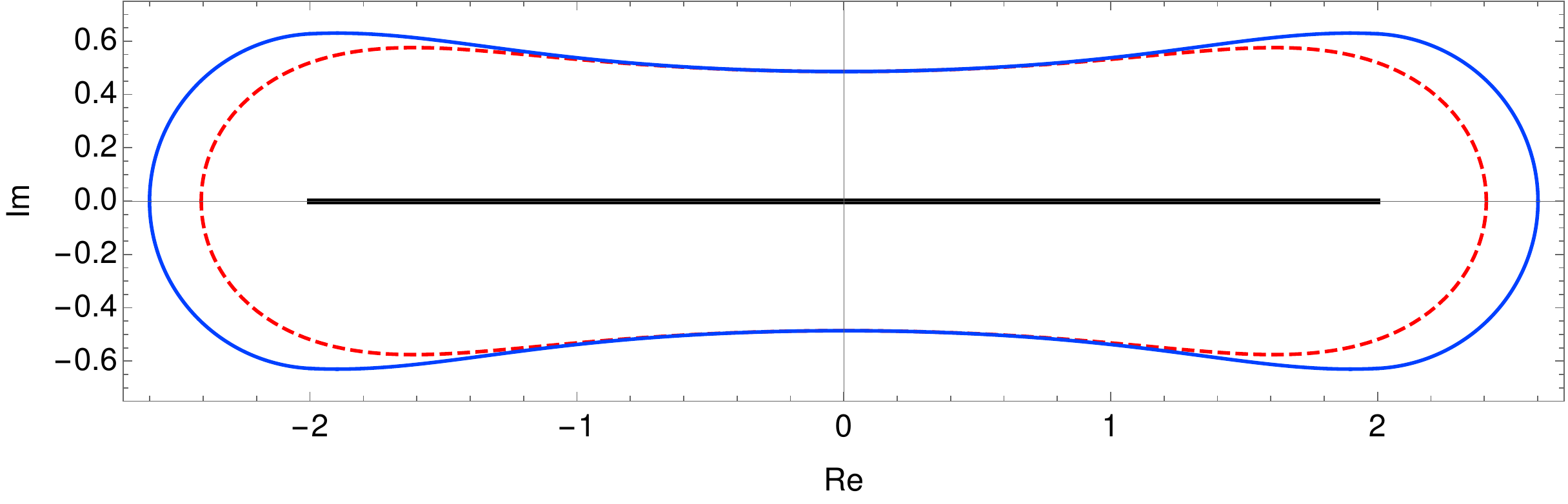}
	\caption{$p=2$}
    \end{subfigure}
\end{figure}
\begin{figure}[H]\ContinuedFloat
    \vskip6pt
    \begin{subfigure}[c]{0.49\textwidth}
	\includegraphics[width=\textwidth]{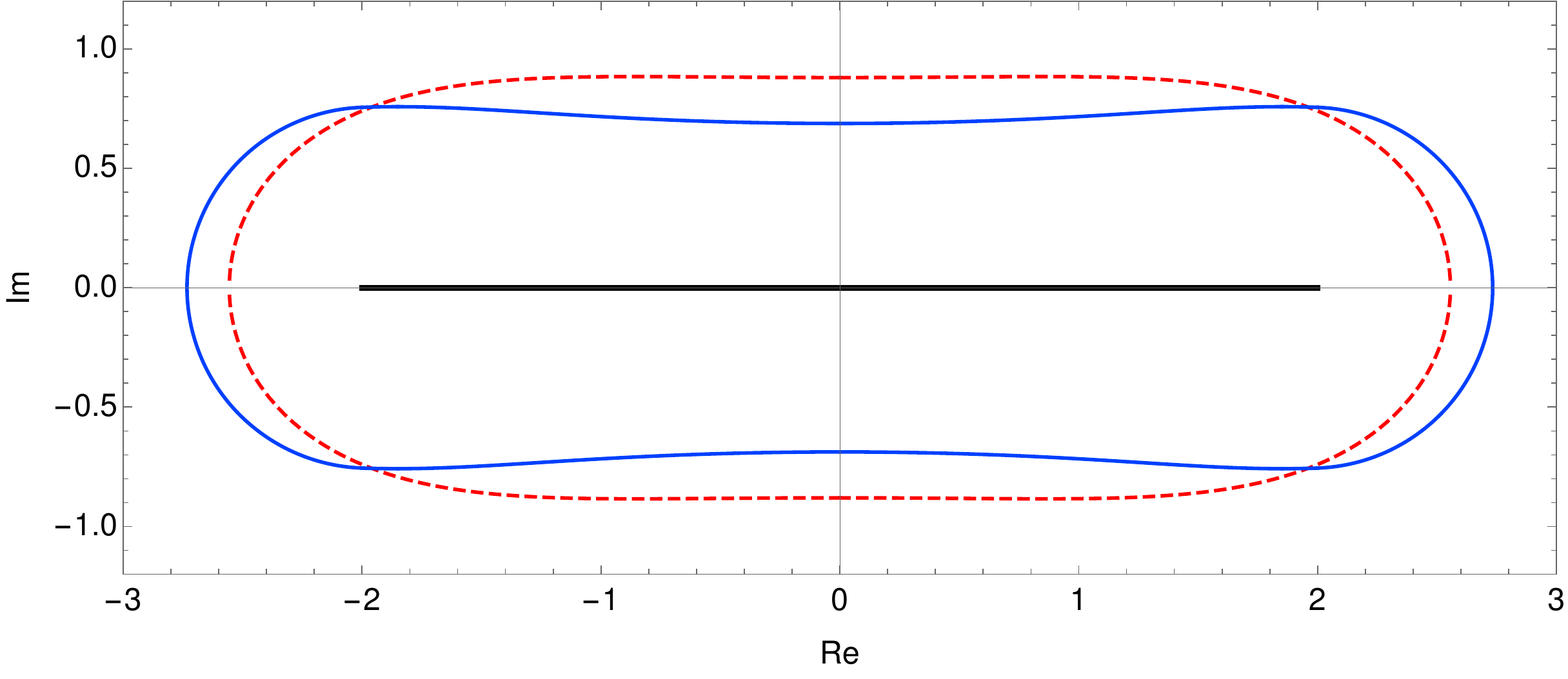}
	\caption{$p=3$}
    \end{subfigure}
    \begin{subfigure}[c]{0.49\textwidth}
	\includegraphics[width=\textwidth]{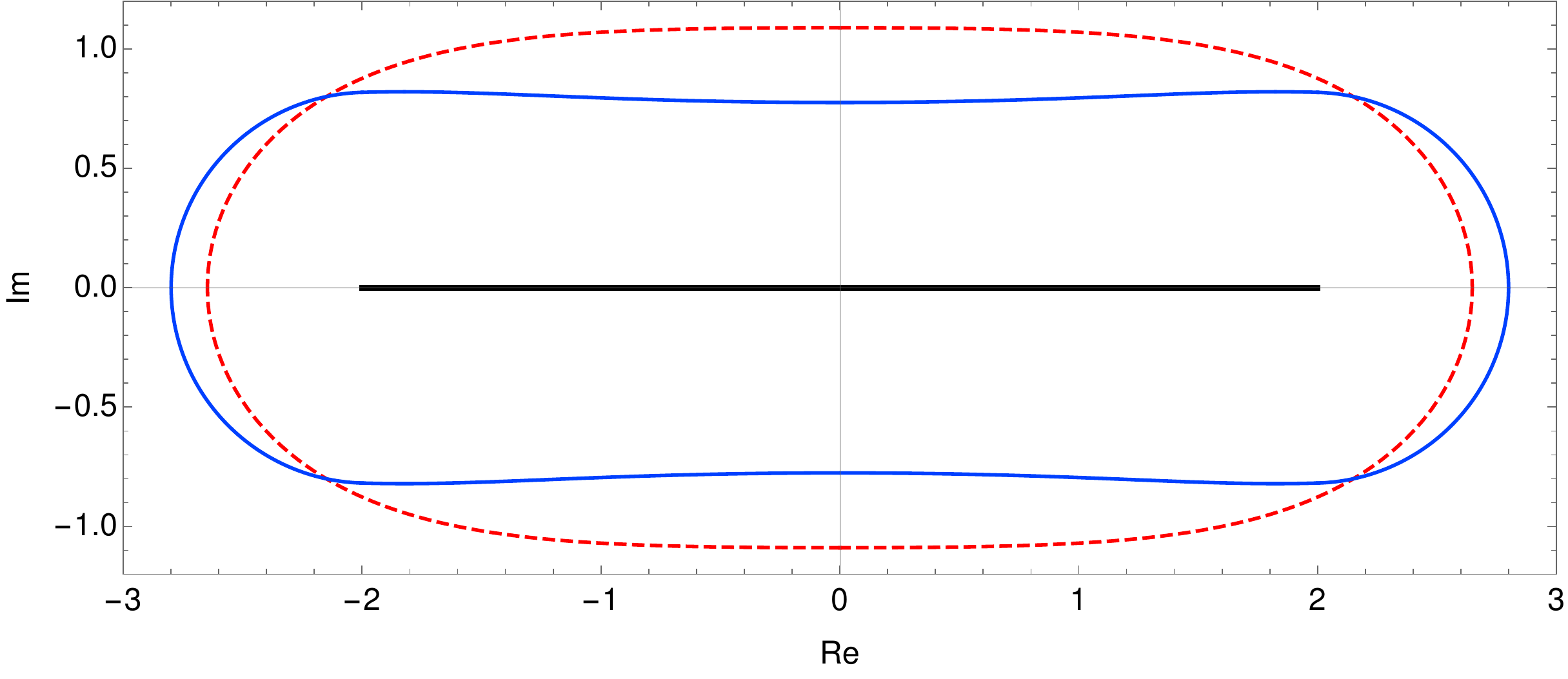}
	\caption{$p=4$}
    \end{subfigure}
    \vskip6pt
    \begin{subfigure}[c]{0.49\textwidth}
	\includegraphics[width=\textwidth]{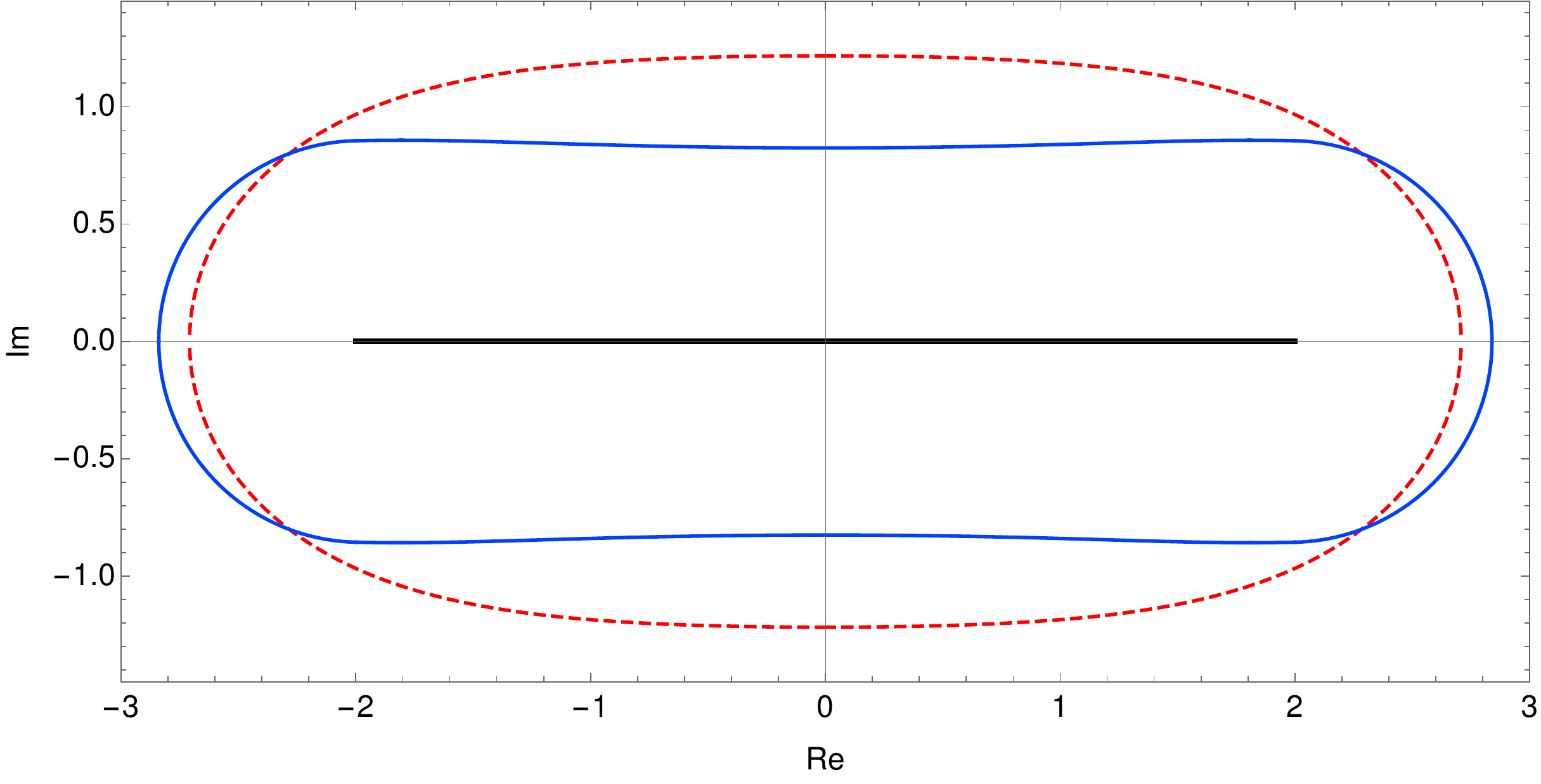}
	\caption{$p=5$}
    \end{subfigure}
    \begin{subfigure}[c]{0.49\textwidth}
	\includegraphics[width=\textwidth]{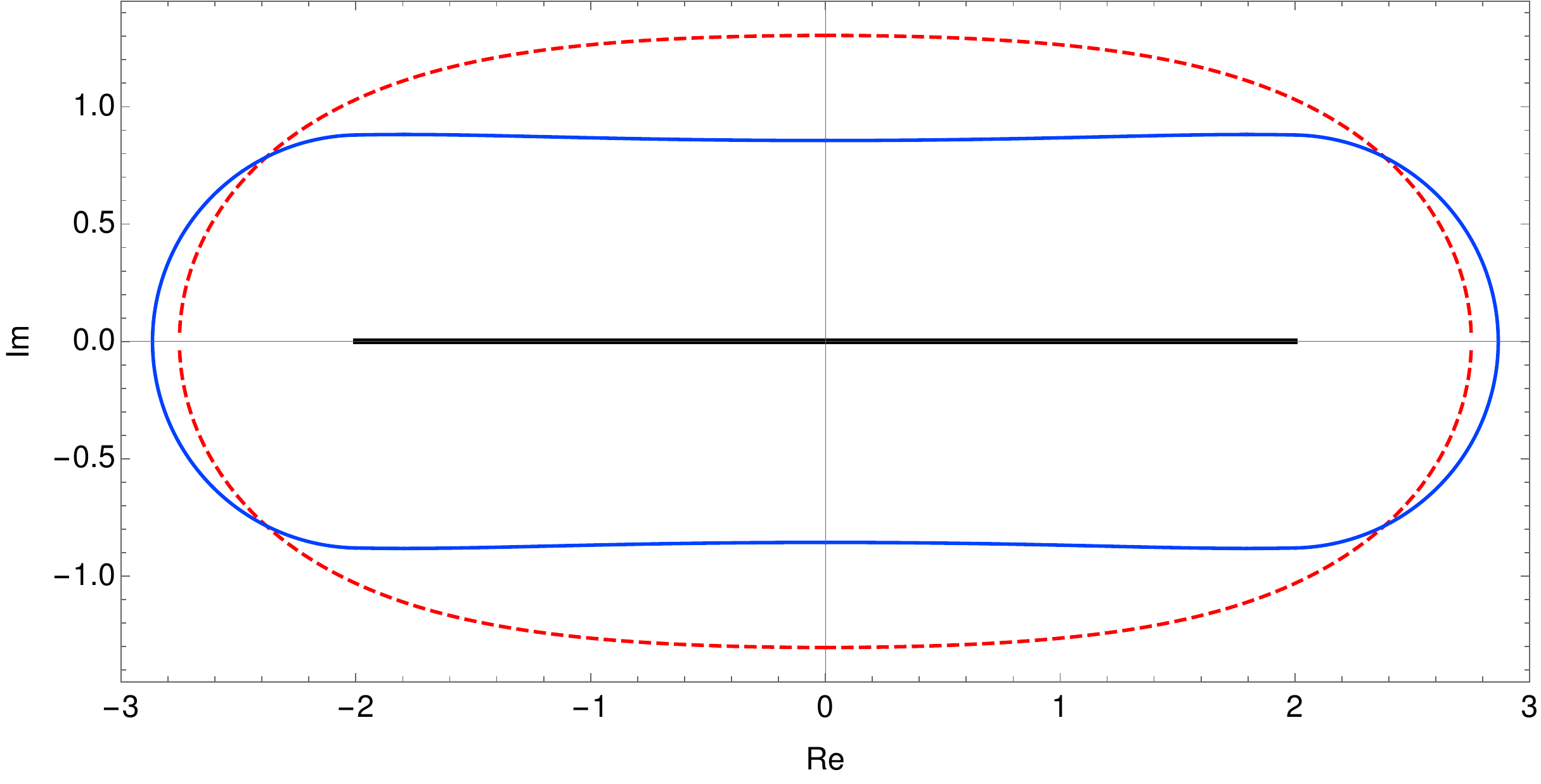}
	\caption{$p=6$}
    \end{subfigure}
    \caption{\small{A comparison of the boundary curves of the spectral enclosure~
			\eqref{eq:p-norm} (red dashed line) with the spectral enclosure~\eqref{eq:p-norm.2} (blue solid line) for six choices of the parameter~$p$ and $\|\ups\|_{\ell^{p}(\ZZ)}=1$.}}
    \label{fig:compar}
\end{figure}

\section*{Acknowledgement}
\noindent
The first author thanks Prof.~David Krej\v{c}i\v{r}{\'\i}k for stimulating discussions. The second author acknowledges financial support by the Ministry of Education, Youth and Sports of the Czech Republic project no. CZ.02.1.01/0.0/0.0/16\_019/0000778.

%=============================================================%

\bibliographystyle{amsplain}
\bibliography{bib_disc_schr}

\end{document}